\documentclass{amsart}

\usepackage{amscd}
\usepackage{amssymb}

\usepackage{amsthm}

\usepackage[mathscr]{eucal}

\usepackage{url}

\newtheorem{thm}{Theorem}
\newtheorem{lem}[thm]{Lemma}
\newtheorem{cor}[thm]{Corollary}
\newtheorem{prop}[thm]{Proposition}

\newtheorem*{thm*}{Theorem}
\newtheorem*{prop*}{Proposition}
\newtheorem*{cor*}{Corollary}
\newtheorem*{lem*}{lem}
\newtheorem*{MT*}{Main Theorem}

\theoremstyle{definition} %
\newtheorem{defn}[thm]{Definition}
\newtheorem*{defn*}{Definition}

\newtheorem{eg}[thm]{Example}

\theoremstyle{remark} %
\newtheorem{rmk}[thm]{Remark}

\newtheorem*{rmk*}{Remark}
\newtheorem*{rmks*}{Remarks}

\numberwithin{equation}{section}
\numberwithin{thm}{section}

\newcommand{\tA}{\mathsf{A}}
\newcommand{\tB}{\mathsf{B}}
\newcommand{\tC}{\mathsf{C}}
\newcommand{\tD}{\mathsf{D}}
\newcommand{\tE}{\mathsf{E}}

\newcommand{\nat}[1]{#1^\natural}

\newcommand{\R}{\mathbb{R}}
\newcommand{\C}{\mathbb{C}}
\newcommand{\hh}{\mathbb{H}}

\newcommand{\G}{\mathscr{G}}

\newcommand{\Gt}{\widetilde{G}}
\newcommand{\Wt}{\widetilde{W}}

\renewcommand{\th}{\theta}

\DeclareMathOperator{\End}{End}

\DeclareMathOperator{\Id}{Id}
\DeclareMathOperator{\tr}{tr}

\DeclareMathOperator{\Ker}{Ker}

\newcommand{\ot}{\otimes}
\newcommand{\qform}[1]{{\left\langle{#1}\right\rangle}}                   

\DeclareMathOperator{\SO}{SO}
\DeclareMathOperator{\Sp}{Sp}
\DeclareMathOperator{\PSp}{PSp}

\DeclareMathOperator{\Spin}{Spin}
\DeclareMathOperator{\SU}{SU}
\DeclareMathOperator{\GL}{GL}

\newcommand{\s}{\sigma}
\newcommand{\Es}{(E, \sigma)}

\newcommand{\ra}{\rightarrow}

\newcommand{\Kalg}{\overline{K}}

\renewcommand{\ge}{\geqslant}

\DeclareMathOperator{\rk}{rk}

\newcommand{\X}{\mathfrak{X}}

\makeatletter
\newenvironment{speceqn}[1]%
{}%
{\addtocounter{equation}{-1}\global\@ignoretrue}
\makeatother

\newcommand{\Z}{\mathbb{Z}}
\newcommand{\Q}{\mathbb{Q}}

\newcommand{\sC}{\mathscr{C}}

\begin{document}

\title[Weakly commensurable subgroups in groups of types $\tB$ and $\tC$]{Weakly
commensurable $S$-arithmetic subgroups in almost
simple  algebraic groups of types $\tB$ and $\tC$}
\author{Skip Garibaldi}
\address{Department of Mathematics and Computer Science, MSC W401, 400 Dowman Dr., Emory University, Atlanta, GA 30322, USA}
\email{skip \text{at} member.ams.org}

\author{Andrei S. Rapinchuk}
\address{Department of Mathematics, University of Virginia, Charlottesville, VA 22904-4137, USA}
\email{asr3x \text{at} virginia.edu}


\dedicatory{To Kevin McCrimmon on the occasion of his retirement}

\begin{abstract}
Let $G_1$ and $G_2$ be absolutely almost simple algebraic groups of
types $\tB_{\ell}$ and $\tC_{\ell}$ 
respectively defined over a number field $K$. We determine when
$G_1$ and $G_2$ have the same isomorphism or isogeny classes of maximal
$K$-tori. This leads to the necessary and sufficient conditions for
two Zariski-dense $S$-arithmetic subgroups of $G_1$ and $G_2$ to be
weakly commensurable.
\end{abstract}

\maketitle

\nocite{PrRap:fields}

\setcounter{tocdepth}{1}
\tableofcontents

\section{Introduction and the statement of main results}\label{S:I}

This paper has two interrelated goals: first, to complete the
investigation of weak commensurability of $S$-arithmetic subgroups
of almost simple algebraic groups begun in \cite{PrRap:weakly}, and second, to
contribute to the classical problem of characterizing almost simple
algebraic groups having the same isomorphism or the same isogeny classes of
maximal tori over the field of definition.

Let $G_1$ and $G_2$ be two semi-simple algebraic groups over a field
$F$ of characteristic zero, and let $\Gamma_i \subset G_i(F)$ be a
(finitely generated) Zariski-dense subgroup for $i = 1, 2$. We recall in \S\ref{S:WC} below the notion 
of {\it weak
commensurability} of $\Gamma_1$ and $\Gamma_2$ introduced in \cite{PrRap:weakly}. 
(This notion was inspired by some problems dealing with
isospectral and length-commensurable locally symmetric spaces, and
we state some geometric consequences
of our main results in \eqref{geom1} and \eqref{geom2}.)  We further recall that the mere existence
of Zariski-dense weakly commensurable subgroups
implies that $G_1$ and $G_2$ either have the same Killing-Cartan
type, or one of them is of type $\tB_{\ell}$ and the other is of
type $\tC_{\ell}.$
Moreover, cumulatively the results of \cite{PrRap:weakly},
\cite{PrRap:Deven} and \cite{G:outer} give, by and large, a complete
picture of weak commensurability for $S$-arithmetic subgroups of
almost simple algebraic groups having the {\it same} type.

On the other hand, weak commensurability of $S$-arithmetic subgroups
in the case where $G_1$ is of type $\tB_{\ell}$ and $G_2$ is of type
$\tC_{\ell}$ has not been investigated so far---it was only pointed
out in \cite{PrRap:weakly} that $S$-arithmetic subgroups
corresponding to the split forms of such groups are indeed weakly
commensurable (see also Remark \ref{R:same} below). Our first
theorem provides a complete characterization of the situations where
$S$-arithmetic subgroups in the groups of types $\tB$ and $\tC$ are
weakly commensurable. In its formulation we will employ the
description, introduced in \cite[\S1]{PrRap:weakly}, of
$S$-arithmetic subgroups of $G(F),$ where $G$ is an absolutely
almost simple algebraic group over a field $F$ of characteristic
zero, in terms of triples $(\G, K, S)$ consisting of a number field
$K \subset F,$ a finite subset $S$  of places of $K$, and an
$F/K$-form $\G$ of the adjoint group $\overline{G}$ --- we briefly recall this description in \S\ref{S:BC}.

\medskip

The following definition will enable us to streamline the statements
of our results.

\begin{defn} \label{twin.def}
Let $\G_1$ and $\G_2$ be absolutely almost simple algebraic groups
of types $\tB_{\ell}$ and $\tC_{\ell}$ with $\ell \ge 2$,
respectively, over a number field $K$. We say that $\G_1$ and $\G_2$
are {\it twins} (over $K$) if for each place $v$ of $K$, both groups
are simultaneously either split or anisotropic over the completion
$K_v$.
\end{defn}

\begin{thm}\label{T:1}
Let $G_1$ and $G_2$ be absolutely almost simple algebraic groups
over a field $F$ of characteristic zero having Killing-Cartan types
$\tB_{\ell}$ and $\tC_{\ell}$ $(\ell \geqslant 3)$ respectively, and
let $\Gamma_i$ be a Zariski-dense $(\G_i, K, S)$-arithmetic subgroup
of $G_i(F)$ for $i = 1, 2.$ Then $\Gamma_1$ and $\Gamma_2$ are
weakly commensurable if and only if the groups $\G_1$ and $\G_2$ are
twins.
\end{thm}

If Zariski-dense $(\G_1, K_1, S_1)$- and $(\G_2, K_2,
S_2)$-subgroups are weakly commensurable then necessarily $K_1 =
K_2$ and $S_1 = S_2$ by \cite[Th.~3]{PrRap:weakly}, so
Theorem \ref{T:1} in fact treats the most general situation.
Furthermore, for $\ell = 2$ we have $\tB_2 = \tC_2$, so $G_1$
and $G_2$ have the same type; then $\Gamma_1$ and $\Gamma_2$ are weakly
commensurable if and only if $\G_1 \simeq \G_2$ over $K$ by \cite[Th.~4]{PrRap:weakly}.  This shows that
the assumption $\ell \geqslant 3$ in Theorem \ref{T:1} is
essential---the excluded case of $\ell = 2$ is treated in Theorem \ref{rank2} below.

\medskip

Turning to the second problem, of characterizing almost simple
algebraic groups having the same (isomorphic classes of) maximal
tori, we would like to point out that, as we will see shortly, one
gets more satisfactory results if instead of talking about {\it
isomorphic} groups one talks about {\it isogenous} ones. We recall
that algebraic $K$-groups $H_1$ and $H_2$  are called isogenous
if there exists a $K$-group $H$ with  central $K$-isogenies
$\pi_i \colon H \to H_i$, $i = 1, 2$. For  semi-simple $K$-groups
$G_1$ and $G_2$, this amounts to the fact that the universal covers
$\Gt_1$ and $\Gt_2$ are $K$-isomorphic, and for 
$K$-tori $T_1$ and $T_2$ this simply means that there exists a
$K$-isogeny $T_1 \to T_2$. Furthermore, we say that two semi-simple
$K$-groups $G_1$ and $G_2$ have the same isogeny classes of maximal
$K$-tori if every maximal $K$-torus $T_1$ of $G_1$ is $K$-isogenous
to some maximal $K$-torus $T_2$ of $G_2$, and vice versa. Unsurprisingly,
$K$-isogenous groups have the same isogeny classes of maximal tori. Using the
results from \cite{PrRap:weakly} and \cite{G:outer}, we prove the following partial
converse for almost simple groups over number fields.

\begin{prop} \label{tori.old}
Let $G_1$ and $G_2$ be  absolutely almost simple algebraic
groups over a number field $K$. Assume that $G_1$ and $G_2$ have the
same isogeny classes of maximal $K$-tori. Then at least one of the
following holds:

\begin{enumerate}
\item \label{P1} $G_1$ and $G_2$ are $K$-isogenous;
\item \label{P2} $G_1$ and $G_2$ are of the same
Killing-Cartan type, which is one of the following: $\tA_{\ell}$
$(\ell > 1)$, $\tD_{2\ell + 1}$ $(\ell > 1)$, or $\tE_6$;
\item \label{P3} one of the groups is of type $\tB_{\ell}$ and
the other of type $\tC_{\ell}$ for some $\ell \geqslant 3$.
\end{enumerate}
\end{prop}

We will  prove the proposition in \S\ref{S:Compl}. As Theorem \ref{rank2}
below shows, it is possible for two isogenous, but not isomorphic,
groups to have the same isomorphism classes of maximal $K$-tori, so
the conclusion in \eqref{P1} cannot be strengthened even if we
assume that $G_1$ and $G_2$ have the same maximal tori. On the other
hand, for each of the types listed in \eqref{P2} one can construct
non-isomorphic simply connected, hence non-isogenous, groups of this
type having the same tori \cite[\S9]{PrRap:weakly}, so these types
are genuine exceptions. In this paper, we will sharpen case
\eqref{P3}. Specifically, we prove the following in \S\ref{S:BC}.

\begin{thm} \label{BC}
Let $G_1$ and $G_2$ be absolutely almost simple algebraic groups
over a number field $K$ of types $\tB_{\ell}$ and $\tC_{\ell}$
respectively for some $\ell \ge 3$.
\begin{enumerate}
\item The groups $G_1$ and $G_2$ have the same \underline{isogeny} classes of
maximal $K$-tori if and only if they are twins.
\item The groups $G_1$ and $G_2$ have the same \underline{isomorphism} classes of maximal $K$-tori if and
only if they are twins, $G_1$ is adjoint, and $G_2$ is simply connected.
\end{enumerate}
\end{thm}

We note that one can give examples of groups $G_1$ and $G_2$ of
types $\tB_{\ell}$ and $\tC_{\ell}$ respectively over the field $\R$
of real numbers, that are neither split nor anisotropic but
nevertheless have the same isomorphism classes of maximal $\R$-tori
(see Example \ref{R.eg}). This shows Theorem \ref{BC}, unlike many
statements about algebraic groups over number fields, is \emph{not} a
global version of the corresponding theorem over local fields. What
is crucial for the proof of Theorem \ref{BC} (and also Theorem
\ref{T:1}) is that if the real groups $G_1$ and $G_2$ are neither
split nor anisotropic with $G_1$ adjoint and $G_2$ simply connected
then they cannot have the same maximal $\mathbb{R}$-tori (see
Corollary \ref{C:3.4}). 

\subsection*{The special case $\tB_2 = \tC_2$}
Theorem \ref{BC} completely settles the question of when the groups
of types $\tB_{\ell}$ and $\tC_{\ell}$ have isogenous tori for $\ell
\ge 3$. The case where $\ell = 2$ is special because the root
systems $\tB_2$ and $\tC_2$ are the same.

Let $G_1$ and $G_2$ be groups of type $\tB_2 = \tC_2$.
They have the same isogeny  classes of maximal tori if and only
if they are isogenous by Lemma \ref{odd.similar} below or \cite[Th.~7.5(2)]{PrRap:weakly}.
In particular, when $G_1$ and $G_2$ are both adjoint or both simply connected, they have
the same isogeny classes of maximal tori if and only if $G_1 \simeq G_2$ if and only if they
have the same maximal tori.  It remains only to give a condition for $G_1$ and $G_2$ to have
the same maximal tori when one is adjoint and the other is simply connected, which we now do.
\begin{thm} \label{rank2}
Let $q_1, q_2$ be $5$-dimensional quadratic forms over a number
field $K$.  The groups $G_1 = \SO(q_1)$ and  $G_2 = \Spin(q_2)$
have the same isomorphism classes of maximal $K$-tori if and only if
\begin{enumerate}
\item \label{rank2.sim} $q_1$ is similar to $q_2$; and
\item \label{rank2.di} $q_1$ and $q_2$ are either both  split or both anisotropic at every completion of $K$.
\end{enumerate}
\end{thm}

\subsection*{Notation} For a number field $K$, we let $V^K$
denote the set of all places, and let $V^K_{\infty}$ (resp.,
$V^K_f$) denote the subset of archimedean (resp., nonarchimedean)
places. Given a reductive algebraic group $G$ defined over a field
$K$, for any field extension $L/K$ we let $\rk_L G$ denote
the $L$-rank of $G$, i.e., the dimension of a maximal $L$-split
torus.

We write $r\qform{a}$ for the symmetric bilinear form $(x, y) \mapsto a \sum_{i=1}^r x_i y_i$ on $K^r$, and adopt similar notation for quadratic forms and hermitian forms.

In \S\ref{S:BC}, we systematically use the following: for $G_1$ and $G_2$  absolutely
almost simple groups of types $\tB_{\ell}$ and $\tC_{\ell}$
 respectively, we put $\nat{G}_1$ for the
adjoint group of $G_1$ (``$\SO$''), and  $\nat{G}_2$ for the simply connected cover
of $G_2$ (``$\Sp$'').

\section{Steinberg's theorem for algebras with involution}\label{Stbg.sec}

Our proofs of Theorems \ref{T:1} and \ref{BC} rely on the
well-known fact that groups of classical types can be realized as
special unitary groups associated with simple algebras with
involutions, so their maximal tori correspond to certain
commutative \'etale subalgebras invariant under the involution. This
description  enables us to apply the local-global principles for the
existence of an embedding of an \'etale algebra with an involutory
automorphism into a simple algebra with an involution
\cite{PrRap:Deven}. To ensure the existence of local embeddings, we
will use an analogue for algebras with involution of the theorem, due
to Steinberg \cite{St:reg}, asserting that if $G_0$ is a quasi-split
simply connected almost simple algebraic group over a field $K$ and
$G$ is an inner form of $G_0$ over $K$ then any maximal $K$-torus
$T$ of $G$ admits a $K$-defined embedding into $G_0$. The required
analogue roughly states that if $(A , \tau)$ is an algebra with
involution such that the corresponding group is quasi-split then any
commutative  \'etale algebra with involution $(E , \sigma)$ that can
potentially embed in $(A , \tau)$ does embed. It can be deduced
from the original Steinberg's theorem along the lines of
\cite[Prop.~3.2(b)]{Gille:tori}, but in fact one can give a simple
direct argument. To our knowledge, this has not been recorded in
the literature.  Further, the argument for type $\tB_n$ (in Proposition \ref{B.split}) extends with minor modifications to other types.
So, despite the fact that we will only
use this statement for algebras corresponding to groups of type
$\tB_n$ and $\tC_n$, we will give the argument for all classical
types. We begin by briefly recalling the types of algebras
with involution arising in this context, indicating in each case the
\'etale subalgebras that give maximal tori.

\subsection*{Description of tori in terms of \'etale algebras}
Let $A$ be a central simple algebra of dimension $n^2$ over a field
$L$ of characteristic $\neq 2$,  and let $\tau$ be an involution of
$A$. Set $K = L^{\tau}$. We recall that $\tau$ is said to be of the
\emph{first} (resp., \emph{second}) kind if the restriction $\tau
\vert_L$ is trivial (resp., nontrivial). Furthermore, if $\tau$ is
an involution of the first kind, then it is either \emph{symplectic}
(i.e., $\dim_K A^{\tau} = n(n-1)/2$) or \emph{orthogonal} (i.e.,
$\dim_K A^{\tau} = n(n+1)/2$).

We also recall the well-known
correspondence between involutions on $A = M_n(L)$ and nondegenerate
hermitian or skew-hermitian forms on $L^n$, cf.~\cite{KMRT}: given
such a form $f$, there exists a unique involution $\tau_f$ such that
$$
f(ax , y) = f(x , \tau_f(a)y)
$$
for all $x , y \in L^n$ and all $a \in A$; then the pair $(M_n(L) ,
\tau_f)$ will be denoted by $A_f$. Moreover, $f$ is symmetric
(resp., skew-symmetric) if and only if $\tau_f$ is orthogonal
(resp., symplectic). Conversely, for any involution $\tau$ there
exists a form $f$ on $L^{n}$ of appropriate type such that $\tau =
\tau_f$, and any two such forms are proportional. (Note that for
involutions of the second kind one can pick the corresponding form to be
either hermitian or skew-hermitian as desired.)

\subsubsection*{Type $^{2}\!\tA_{\ell}$} Let $(A , \tau)$ be a central
simple $L$-algebra of dimension $n^2$ with an involution $\tau$ of
the second kind. Then $G = \mathrm{SU}(A , \tau)$ is an
absolutely almost simple simply connected $K$-group of type
$^{2}\!\tA_{\ell}$ with $\ell = n - 1$, and conversely any such
group corresponds to an algebra with involution $(A , \tau)$ of this
kind. Any $\tau$-invariant \'etale commutative subalgebra $E \subset
A$ gives  a maximal $K$-torus
\[
T = \mathrm{R}_{E/K}(\GL_1) \cap G = \mathrm{SU}(E , \tau \vert_E)
\]
of $G$, and all maximal $K$-tori are obtained this way (see, for example, \cite[Prop.~2.3]{PrRap:Deven}). The
group $G$ is quasi-split if and only if $A = M_n(L)$ and $\tau =
\tau_h$ where $h$ is a nondegenerate hermitian form on $L^n$ of Witt
index $[n/2]$.

\subsubsection*{Type $\tB_{\ell}$ $(\ell \geqslant 2)$}  Let $A =
M_n(K)$ with $n = 2\ell + 1$, and let $\tau$ be an orthogonal
involution of $A$. Then $\tau = \tau_f$ for some nondegenerate
symmetric bilinear form $f$ on $K^n$, and
$$
G = \mathrm{SU}(A , \tau) = \mathrm{SO}(f)
$$
is an adjoint group of type $\tB_{\ell}$, and every such group is
obtained this way. Furthermore, maximal $K$-tori $T$ of $G$
bijectively correspond to maximal commutative \'etale
$\tau$-invariant subalgebras $E$ of $A$ (of dimension $n$) such that
$\dim_K E^{\tau} = \ell + 1$ under the correspondence given by
$$
T = \mathrm{R}_{E/K}(\GL_1) \cap G  = \mathrm{SU}(E , \tau
\vert_E).
$$
Furthermore, any
such algebra admits a decomposition
\begin{equation}\label{E:decomp}
(E , \tau) = (E' , \tau') \times (K , \mathrm{id}_K)
\end{equation}
where $E' \subset E$ is a $\tau$-invariant subalgebra of dimension
$2\ell$. Finally, the group $G$ is quasi-split (in fact, split) if
and only if $f$ has Witt index $\ell$.

\subsubsection*{Type $\tC_{\ell}$ $(\ell \geqslant 2)$} Let $A$ be a
central simple $K$-algebra of dimension $n^2$ with $n = 2\ell$, and
let $\tau$ be a symplectic involution of $A$. Then $G =
\mathrm{SU}(A , \tau)$ is an absolutely almost simple simply
connected group of type $\tC_{\ell}$, and all such groups are
obtained this way. Maximal $K$-tori of $G$ correspond to maximal
commutative \'etale $\tau$-invariant subalgebras $E \subset A$ (of
dimension $n$) such that $\dim_K E^{\tau} = \ell$ in the fashion
described above. The group $G$ is quasi-split (in fact, split) if
and only if $A = M_n(K)$. Then $\tau = \tau_f$ where $f$ is a
nondegenerate skew-symmetric form on $K^n$; there is only one
equivalence class of such forms, so in this case $G \simeq
\mathrm{Sp}_n$.

\subsubsection*{Type $^{1,2}\tD_{\ell}$ $(\ell \geqslant 4)$} 
Let $A$
be a central simple $K$-algebra $K$-algebra of dimension $n^2$ where
$n = 2\ell$, and let $\tau$ be an orthogonal involution of $A$. Then
$G = \mathrm{SU}(A , \tau)$ is an almost absolutely simple $K$-group
of type $^{1,2}\tD_{\ell}$ which is neither simply connected nor
adjoint, and any $K$-group of this type is $K$-isogenous to such a
group. Maximal $K$-tori of $G$ correspond to maximal commutative
\'etale $\tau$-invariant subalgebras $E \subset A$ (of dimension
$n^2$) such that $\dim_K E^{\tau} = \ell$. The group $G$ is
quasi-split if and only if $A = M_n(K)$ and $\tau = \tau_f$ where
$f$ is a symmetric bilinear form on $K^n$ of Witt index $\ell - 1$
or $\ell$.

\subsubsection*{Summary}
Thus, if $A$ is a central simple $L$-algebra of dimension $n^2$ (and
$L = K$ for all types except $^{2}\!\tA_{\ell}$) then maximal
$K$-tori of the algebraic $K$-group $G = \mathrm{SU}(A , \tau)$
correspond in the manner described above to maximal abelian \'etale
$\tau$-invariant subalgebras $E \subset A$ with $\dim_L E = n$ such
that for $\sigma = \tau \vert_E$ we have
\begin{equation}\label{E:dimen}
\dim_K E^{\sigma} = \begin{cases} n & \text{if $\sigma \vert_L \neq \mathrm{id}_L$} \\
 \left[ \frac{n+1}{2} \right] & \text{if $\sigma \vert_L = \mathrm{id}_L$.}
\end{cases}
\end{equation}
(Note that the condition is automatically satisfied if $\sigma \vert_L \neq \mathrm{id}_L$.) 

Now, let $(E , \sigma)$ be an $n$-dimensional commutative \'etale
$L$-algebra with an involution satisfying (\ref{E:dimen}). Then the
question of whether the $K$-torus $T = \mathrm{SU}(E ,
\sigma)^{\circ}$ can be embedded into $G = \mathrm{SU}(A , \tau)$
where $A$ is a central simple $L$-algebra of dimension $n^2$ with an
involution $\tau$ such that $\sigma \vert_L = \tau \vert_L$
translates into the question if there is an embedding $(E , \sigma)
\hookrightarrow (A , \tau)$ of $L$-algebras with involution, which
we will now investigate in the cases of interest to us. We note that
if $G$ is quasi-split then in all cases $A = M_n(L)$. In this case,
the universal way to construct an embedding $(E , \sigma)
\hookrightarrow (M_n(L) , \tau)$ is described in the following
well-known statement.
\begin{prop}\label{P:exist}
Let $(E , \sigma)$ be an $n$-dimensional commutative \'etale
$L$-algebra with an involution $\sigma$.
\begin{enumerate}
\renewcommand{\theenumi}{\roman{enumi}}
\item For any $b \in E^{\times}$, the map
$\phi_b \colon E \times E \to K$ given by $$\phi_b(x , y) =
\mathrm{tr}_{E/L}(x \cdot b \cdot \sigma(y))$$ is a nondegenerate
sesqui-linear form, which is hermitian or skew-hermitian if and only
if $b$ is such.

\item Let $b \in E^{\times}$ be hermitian or
skew-hermitian, and let $\tau_{\phi_b}$ be the involution on $A :=
\End_L(E) \simeq M_n(L)$ corresponding to $\phi_b$; then the
regular representation of $E$ gives an embedding $$(E , \sigma)
\hookrightarrow (A , \tau_{\phi_b}) = A_{\phi_b}$$ of algebras with
involution.

\item Let $\tau$ be an involution on $A =
M_n(L)$, and let $f$ be a hermitian or skew-hermitian form on $L^n$
such that $\tau_f = \tau$. Then the following conditions are
equivalent:
\begin{enumerate}
\renewcommand{\theenumii}{\alph{enumii}}
\item \label{3a} There exists $b
\in E^{\times}$ of the same type as $f$ such that $\phi_b$ is
equivalent \newline to $f$.

\item \label{3b} There exists a form $h$ on $E \simeq L^n$ which is equivalent
to $f$ and satisfies
\begin{equation}\label{E:h}
h(ax , y) = h(x , \sigma(a)y) \ \ \text{for all}  \ \ a, x, y \in E.
\end{equation}
\item \label{3c} There exists an embedding $(E
, \sigma) \hookrightarrow (A , \tau)$ as $L$-algebras with
involutions.
\end{enumerate}
\end{enumerate}
\end{prop}

\begin{proof}[Sketch of proof]
The nondegeneracy of $\phi_b$ in (i)  follows from the fact that
the $L$-bilinear form on $E$ given by $(x , y) \mapsto
\tr_{E/L}(xy)$ is nondegenerate as $E/L$ is \'etale; other
assertions in (i) and (ii) are immediate consequences of the
definitions. The implications (a) $\Rightarrow$ (b) $\Rightarrow$ (c)
in (iii) are obvious, and the equivalence (a) $\Leftrightarrow$
(c) (which we will not need) is established in \cite[Prop.~7.1]{PrRap:Deven}.
\end{proof}

We also note that in fact any nondegenerate
hermitian/skew-hermitian form $h$ on $E$ satisfying \eqref{E:h} is
of the form $\phi_b$ for some $b \in E^{\times}$ of the respective
type. Indeed, since the form $\phi_1$ is nondegenerate, we can write
$h$ in the form $h(x , y) = \mathrm{tr}_{E/L}(x\cdot g(\sigma(y)))$ for
some $K$-linear automorphism $g$ of $E$. Then \eqref{E:h} implies
that $g$ is $E$-linear, and therefore is of the form $g(x) = bx$ for
some $b \in E^{\times}$, which will necessarily be of appropriate
type.

\begin{eg}[involutions of the first kind] \label{quad.etale} 
According to Proposition 2.2 in
\cite{PrRap:Deven}, if $L = K$ and $(E , \sigma)$ is a $K$-algebra
with involution of dimension $n = 2\ell$ satisfying \eqref{E:dimen}
then
$$
(E , \sigma) \simeq (F[\delta]/(\delta^2 - d) , \theta)
$$
where $F = E^{\sigma}$, $d \in F^{\times}$, and $\theta(\delta) = -\delta$.

For invertible $b \in E^{\sigma}$ and $x_i, y_i \in F$, we have 
\[
\phi_b(x_1 + y_1\delta , x_2+y_2\delta) = \tr_{E/K}(bx_1x_2 - bdy_1y_2) =  \tr_{F/K}(2b (x_1 x_2 - d y_1 y_2)),
\]
so $\phi_b$ is the transfer from $F$ to $K$ of 
 the symmetric bilinear form $\qform{2b, -2bd}$.  Clearly, if $E$ is $F \times F$, then $\phi_b$ is hyperbolic.
\end{eg}

The example gives the entries in the $\phi_b$ column of Table \ref{R.tori}.

\begin{prop}[type $\tC$] \label{C.split}
Let $\Es$ be an \'etale $K$-algebra of dimension $n = 2\ell$ with
involution satisfying \eqref{E:dimen}.  Then for every symplectic
involution $\tau$ on $M_{n}(K)$, there is a $K$-embedding $\Es
\hookrightarrow (M_{n}(K), \tau)$.
\end{prop}

\begin{proof}
It follows from the structure of $(E , \sigma)$ in the example that
there exists a skew-symmetric \emph{invertible} $b \in E$ (one can
take, for example, the element corresponding to $\delta$); then by
Proposition \ref{P:exist}(i), the form $\phi_b$ is nondegenerate and
skew-symmetric. On the other hand, since $\tau$ is symplectic, we
have $\tau = \tau_f$ for some nondegenerate skew-symmetric form $f$
on $K^{n}$. As any two such forms are equivalent, our assertion
follows from Proposition \ref{P:exist}(iii).
\end{proof}

To handle the algebras corresponding to types $\tB$ and $\tD$, we
need the following.
\begin{lem}\label{L:form}
Let $(E , \sigma)$ be a commutative \'etale $K$-algebra with
involution of dimension $n = 2\ell$ satisfying \eqref{E:dimen}. Then
there exists a nondegenerate symmetric bilinear form $h$ on $E$ that
satisfies \eqref{E:h} and has Witt index $\geqslant \ell - 1$.
\end{lem}
\begin{proof}
If $K$ is finite then one can take, for example, $h = \phi_1$, so we
can assume in the rest of the argument that $K$ is infinite. It
follows from the  description of $E$ that, for $\Kalg$ an algebraic closure of $K$,
$$
(E \otimes_K \overline{K} , \sigma \otimes \mathrm{id}_K) \simeq (M
, \mu)
$$
where $M = \prod_{i = 1}^{\ell} (\overline{K} \times \overline{K})$
and $\mu$ acts on each copy of $\overline{K} \times \overline{K}$ by
switching components. Viewing $M$ as an affine $n$-space, consider
the $K$-defined subvariety  $M_{-} := \{ x \in M \mid \mu(x) = -x
\}$.
Clearly, $M_{-}$ is a $K$-defined vector space,
so the $K$-points $E_{-} := M_{-} \cap E$
are Zariski-dense in $M_{-}$.
On the other hand, let $U \subset M$ be the Zariski-open subvariety
of elements with pairwise distinct components; then any $x \in U$
generates $M$ as a $\overline{K}$-algebra. Furthermore, it is easy to
see that $U \cap M_{-} \neq \emptyset$, so $U \cap E_{-} \neq
\emptyset$.

Fix $e \in U \cap E_{-}$; then $1, e, \ldots , e^{n-1}$ form a
$K$-basis of $E$. For $x \in E$ we define $c_i(x)$ for $i = 0, \ldots ,
n-1$ so that $x = \sum_{i = 0}^{n-1} c_i(x) e^i$. Set
$$
h(x , y) := c_{n-2}(x \sigma(y)).
$$
Clearly, $h$ is symmetric bilinear and satisfies \eqref{E:h}. Let us
show that $h$ is nondegenerate. If $x = \sum_{i = 0}^{n-1} c_i(x)
e^i$ is in the radical  of $h$, then so is $\sigma(x)$, and
therefore also $x_{+} := \sum_{i = 0}^{\ell - 1} c_{2i}(x) e^{2i}$
and $x_{-} := \sum_{i = 0}^{\ell - 1} c_{2i+1}(x)e^{2i+1}$. From
$h(x_{+} , 1) = 0$, $h(x_{+} , e^2) = 0$, etc., we successively
obtain that $c_{n-2}(x) = 0$, $c_{n-4}(x) = 0$, etc., i.e., $x_{+} =
0.$ Furthermore, we have $0 = h(x_{-} , e^{-1}) = - c_{n-1}(x)$. Then
from $h(x_{-} , e) = 0$, $h(x_{-} , e^3) = 0$, etc., we successively
obtain $c_{n-3}(x) = 0$, $c_{n-5}(x) = 0$, etc. Thus, $x_{-} = 0$,
hence $x = 0$, as required.
It remains to observe that the subspace spanned by $1, e, \ldots ,
e^{\ell - 2}$ is totally isotropic with respect to $h$.
\end{proof}

\begin{rmk*}
In an earlier version of this paper, we constructed $h$ in Lemma
\ref{L:form} in the form $h = \phi_b$ using some matrix
computations. The current proof, which minimizes computations, was
inspired by \cite[\S5]{BhGross}.
\end{rmk*}

\begin{prop}[type $\tB$] \label{B.split}
Let $\Es$ be an \'etale $K$-algebra of dimension $n= 2\ell + 1$ with
involution satisfying \eqref{E:dimen}. If $\tau$ is an orthogonal
involution on $A = M_n(K)$ such that $\tau = \tau_f$ where $f$ is a
nondegenerate symmetric bilinear form on $K^n$ of Witt index $\ell$
then there exists an embedding $\Es \hookrightarrow (A , \tau)$ of
$K$-algebras with involution.
\end{prop}

\begin{proof}
Pick a decomposition \eqref{E:decomp}, and then use Lemma
\ref{L:form} to find a form $h'$ on $E'$ with the properties
described therein. We can write $h' = h'_1 \perp h'_2$ where $h'_1$
is a direct sum of $\ell - 1$ hyperbolic planes and $h'_2$ is a
binary form. Choose a 1-dimensional form $h''$ so that $h'_2 \perp
h''$ is isotropic, and consider $h = h' \perp h''$ on $E = E' \times
K$. Then $h$ is a nondegenerate symmetric bilinear form on $E$
satisfying \eqref{E:h} and having Witt index $\ell$. So, $h$ is
equivalent to $f$, hence $(E , \sigma)$ embeds in $(A , \tau)$ by
Proposition \ref{P:exist}(iii).
\end{proof}

\begin{rmk}\label{R:same}
Let now $G_1$ be the $K$-split adjoint group $
\SO_{2\ell+1}$ of type $\tB_{\ell}$
and $G_2$ be the $K$-split simply connected group $\Sp_{2\ell}$ of type
$\tC_{\ell}$ where $\ell \geqslant 2$. It was observed in
\cite{PrRap:weakly}, Example 6.7, that $G_1$ and $G_2$ have the same
isomorphism classes of maximal $K$-tori over any field $K$ of
characteristic $\neq 2$. This was derived from the fact that $G_1$
and $G_2$ have isomorphic Weyl groups using the results of
\cite{Gille:tori} or \cite{Ragh:tori}. Now, we are in a
position to give a much simpler explanation of this phenomenon.
Indeed, $G_1 = \mathrm{SU}(A_1 , \tau_1)$ where $A_1 = M_{2\ell +
1}(K)$ and $\tau_1$ is an orthogonal involution on $A_1$
corresponding to a nondegenerate symmetric bilinear form on
$K^{2\ell + 1}$ of Witt index $\ell$, and $G_2 = \mathrm{SU}(A_2 ,
\tau_2)$ where $A_2 = M_{2\ell}(K)$ and $\tau_2$ is a symplectic
involution on $A_2$ corresponding to a nondegenerate skew-symmetric
form on $K^{2\ell}$. Any maximal $K$-torus $T_2$ of $G_2$ is of the
form $\mathrm{SU}(E_2 , \sigma_2)$ where $E_2$ is a
$2\ell$-dimensional commutative $\tau_2$-invariant subalgebra of
$A_2$, $\sigma_2 = \tau_2 \vert_{E_2}$, with $(E_2 , \sigma_2)$
satisfying \eqref{E:dimen}. Set $(E_1 , \sigma_1) = (E_2 , \sigma_2)
\times (K , \mathrm{id}_K)$. According to Proposition \ref{B.split},
there exists an embedding $(E_1 , \sigma_1) \hookrightarrow (A_1 ,
\tau_1)$, which gives rise to a $K$-isomorphism between $T_2$ and
the maximal $K$-torus $T_1 = \mathrm{SU}(E_1 , \sigma_1)$ of $G_1$.
This, combined with the symmetric argument based on Proposition
\ref{C.split}, yields the required fact. Then, repeating the
argument given in \emph{loc.~cit.}, we conclude that if $K$ is a
number field then for any finite subset $S \subset V^K$ containing
$V^K_{\infty}$, the $S$-arithmetic subgroups of $G_1$ and $G_2$ are
weakly commensurable.
\end{rmk}

Turning now to type $\tD_{\ell}$, we first observe that if $(E ,
\sigma)$ is a $K$-algebra with involution of dimension $n = 2\ell$
satisfying \eqref{E:dimen} then the determinant --- viewed as an element of $K^{\times}/{K^{\times 2}}$ ---
of the symmetric
bilinear form $\phi_b$ for invertible $b \in E^{\sigma}$ does not
depend on $b$ (cf.~\cite[Cor.~4.2]{BKM}) and will be denoted $d(E , \sigma)$. Now, if
$\tau$ is an involution on $A = M_n(K)$ that corresponds to a
symmetric bilinear form $f$ on $K^n$ having determinant $d(f)$ then
it follows from Proposition \ref{P:exist}(iii) that an embedding $(E
, \sigma) \hookrightarrow (A , \tau)$ can exist only if $d(E ,
\sigma) = d(f)$ in $K^{\times}/{K^{\times 2}}$.

\begin{prop}\label{D.qsplit}
Let $(E , \sigma)$ be an \'etale $K$-algebra of dimension $n =
2\ell$ with involution satisfying \eqref{E:dimen}. If $\tau$ is an
orthogonal involution on $A = M_n(K)$ such that $\tau = \tau_f$
where $f$ is a nondegenerate symmetric bilinear form on $K^n$ of
Witt index $\geqslant \ell - 1$ such that $d(E , \sigma) = d(f)$ (in
$K^{\times}/{K^{\times 2}}$) then there exists an embedding $(E ,
\sigma) \hookrightarrow (A , \tau)$ of $K$-algebras with involution.
\end{prop}
\begin{proof}
Let $h$ be the symmetric bilinear form on $E$ constructed in Lemma
\ref{L:form}. As we observed after Proposition \ref{P:exist}, $h$ is
actually of the form $h = \phi_b$ for some invertible $b \in
E^{\sigma}$, so $d(h) = d(E , \sigma)$. We can write $h = h_1 \perp
h_2$ where $h_1$ is a direct sum of $\ell - 1$ hyperbolic planes and
$h_2$ is a binary form. Similarly, $f = f_1 \perp f_2$ where $f_1$
is a direct sum of $\ell - 1$ hyperbolic planes and $f_2$ is binary.
Then $d(E , \sigma) = d(f)$ implies that $d(h_2) = d(f_2)$, so $h_2$
and $f_2$ are similar. Thus, a suitable multiple of $h$ is
equivalent to $f$, and our claim follows from Proposition
\ref{P:exist}(iii).
\end{proof}

Finally, we will treat algebras corresponding to the groups of type
$^{2}\!\tA_{\ell}$. Here $L$ will be a quadratic extension of $K$
and all involutions will restrict to the nontrivial automorphism of
$L/K$.

\begin{prop}[type $\tA$]\label{A.qsplit}
Let $(E , \sigma)$ be an \'etale $n$-dimensional $L$-algebra with
involution. If $\tau$ is a unitary involution on $A = M_n(L)$ such
that $\tau = \tau_f$ where $f$ is a hermitian form on $L^n$ having
Witt index $m := [ n/2 ]$, then there exists an embedding $(E ,
\sigma) \hookrightarrow (A , \tau)$ of $L$-algebras with involution.
\end{prop}
\begin{proof}
It is enough to construct a nondegenerate hermitian form on $E$ that
satisfies \eqref{E:h} and has Witt index $m$. If $K$ is finite, one
can take, for example, $h = \phi_1$, so we can assume that $K$ is
infinite. Set $F = E^{\sigma}$ so that $E = F \otimes_K L$. Since
$K$ is infinite, arguing as in the proof of Lemma \ref{L:form}, one
can find $e \in F$ so that $F = K[e]$. Then any $x \in E$ admits a
unique presentation of the form $x = \sum_{i = 0}^{n-1} e^i \ot c_i(x)$
with $c_i(x) \in L$. Define
$$
h(x , y) := c_{n-1}(x\sigma(y)).
$$
It is easy to see $h$ is a hermitian form satisfying \eqref{E:h};
let us show that it is nondegenerate. If $x$ is in the radical of
$h$ then from $h(x , 1) = 0$, $h(x , e) = 0$, etc., we successively
obtain that $c_{n-1}(x) = 0$, $c_{n-2}(x) = 0$, etc. Thus, $x = 0$,
proving the nondegeneracy of $h$. Since $2(m - 1) < n - 1$, the
subspace spanned by $1, e, \ldots , e^{m-1}$ is totally isotropic, hence the Witt index of $h$ is $m$, as required.
\end{proof}

\section{Maximal tori in real groups of types $\tB$ and
$\tC$}\label{S:real}

This section is devoted to determining the isomorphism classes of
maximal tori in certain linear algebraic groups, primarily of types
$\tB$ and $\tC$, over the real numbers. Recall that every torus $T$
over $\R$ is $\R$-isomorphic to the product
\begin{equation}\label{E:r-tori}
(\GL_1)^{\alpha} \times
(\mathrm{R}^{(1)}_{\C/\R}(\GL_1))^{\beta} \times
(\mathrm{R}_{\C/\R}(\GL_1))^{\gamma}
\end{equation}
for uniquely determined nonnegative integers $\alpha, \beta, \gamma$
\cite[p.~64]{Vos}, and then the group $T(\R)$ is topologically
isomorphic to $(\R^\times)^{\alpha} \times (S^1)^{\beta} \times
(\C^\times)^{\gamma}$, where $S^1$ is the group of complex numbers
of modulus $1$. The fact that $T$ is isomorphic to a maximal
$\R$-torus of a given reductive $\R$-group $G$ typically imposes
serious restrictions on the numbers $\alpha, \beta$ and $\gamma$. To
illustrate this, we first consider the following easy example.

\begin{eg} \label{GLn.eg}
\emph{Every maximal $\R$-torus in $G = \GL_{n , \hh}$, where $\hh$
is the algebra of Hamiltonian quaternions, is isomorphic to
$(\mathrm{R}_{\C/\R}(\GL_1))^n$.} Indeed, every maximal
$\R$-torus in $G$ is of the form $\mathrm{R}_{E/\R}(\GL_1)$
where $E$ is a maximal commutative $2n$-dimensional \'etale
subalgebra of $A = M_n(\hh).$ Any commutative $2n$-dimensional
\'etale $\R$-algebra $E$ is isomorphic to $\R^\alpha \times \C^\gamma$
with $\alpha + 2\gamma = 2n$. But in order for $E$ to have an
$\R$-embedding in $A$, we must have $\alpha = 0$ and then $\gamma =
n$ (cf.~\cite[2.6]{PrRap:Deven}), so our claim follows.
\end{eg}

We now recall the standard notations for some classical real
algebraic groups. We let $\mathrm{SO}(r , n - r)$ denote the special
orthogonal group of the $n$-dimensional quadratic form $q = r\qform{1}  \perp (n - r)\qform{-1}$.
 Similarly, we let
$\mathrm{Sp}(r , n - r)$ denote the special unitary group of the
$n$-dimensional hermitian form $h = r\qform{1} \perp (n - r)
\qform{-1}$ over $\hh$ with the standard involution. Every adjoint $\R$-group of type $\tB_{\ell}$ is
isomorphic to some $\mathrm{SO}(r , n - r)$ for $n = 2\ell + 1$ and
some $0 \leqslant r \leqslant n$, and every nonsplit simply
connected $\R$-group of type $\tC_{\ell}$ is isomorphic to
$\mathrm{Sp}(r , \ell - r)$  some $0 \leqslant r \leqslant \ell$.

\begin{lem}[Adjoint $\tB_{\ell}$ over $\R$] \label{B.R}
The maximal $\R$-tori in  $G = \SO(r , n - r)$, where $n = 2\ell +
1$,  are of the form \eqref{E:r-tori} with  $\alpha + \beta +
2\gamma = \ell$ and $\alpha + 2 \gamma \leqslant s := \min(r , n -
r)$.
\end{lem}

\begin{proof}
Let $\tau$ be the involution on $A = M_n(K)$ that corresponds to the
symmetric bilinear form $f$ associated with the quadratic form $q = r\qform{1} \perp (n-r) \qform{-1}$ so that $G =
\mathrm{SU}(A , \tau)$. Let $T$ be a maximal $\R$-torus of $G$
written in the form \eqref{E:r-tori}. Since the rank of $G$ is
$\ell$, we immediately obtain
$$
\dim T = \alpha + \beta + 2\gamma = \ell.
$$
Furthermore, we have $T = \mathrm{SU}(E , \sigma)$ where $E \subset
A$ is a $\tau$-invariant maximal commutative \'etale subalgebra,
$\sigma = \tau \vert_E$, and \eqref{E:dimen} holds. There are
exactly $4$ isomorphism classes of indecomposable \'etale
$\R$-algebras with involution, which are listed in
Table~\ref{R.tori}. Using this information, we can write
$$
(E , \sigma) = \R^{\delta_1} \times (\R \times \R)^{\delta_2}
\times \C^{\delta_3} \times (\C \times \C)^{\delta_4}
$$
where the involutions on factors are as in the table. Comparing this
with the structure of $T$, we obtain $\delta_2 = \alpha$, $\delta_3 = \beta$, and
$\delta_4 = \gamma$. According to Proposition \ref{P:exist}(iii),
there exists $b \in E^{\sigma}$ such that $\phi_b$ is equivalent to
$f$. But the Witt index of $f$ is $s$ (which equals the $\R$-rank of
$G$), and the Witt index of $\phi_b$ is $\geqslant \delta_2 +
2\delta_4$. Thus, $\alpha + 2\gamma \leqslant s$. (We note that
$\rk_{\R}T = \alpha + \gamma$, immediately yielding the
restriction $\alpha + \gamma \leqslant s$. So, the restriction we
have actually obtained is stronger than one can a priori expect.)

Conversely, suppose $\alpha, \beta, \gamma$ satisfy the two
constraints, and assume that $r > n - r$ (otherwise we can replace
the quadratic form $q$ defining $G$ with $-q$); in particular, $r > \ell$. Consider the \'etale
$\R$-algebra
$$
(E , \sigma) = \R \times (\R \times \R)^{\alpha} \times \C^{\beta}
\times (\C \times \C)^{\gamma} =: (E_1 , \sigma_1) \times \cdots
\times (E_4 , \sigma_4)
$$
of dimension $1+ 2\alpha + 2\beta + 4\gamma  = 2\ell + 1 = n$ where
the involutions on the factors $\R$, $\R \times \R, \ldots $ are as
described in Table \ref{R.tori}. (Clearly, $E$ satisfies
\eqref{E:dimen}.) Let us show that there exists $b = (b_1, \ldots ,
b_4) \in E^{\sigma}$ such that $\phi_b$ is equivalent to $f$. Set
$b_2 = ((1 , 1), \ldots , (1 , 1))$ and $b_4 = ((1 , 1), \ldots , (1
, 1))$. Then the quadratic form associated with the bilinear form
$(\phi_{2 , 4})_{(b_2 , b_4)}$ on $E_2 \times E_4$ is equivalent to
$(\alpha + 2\gamma)(\qform{1} \perp  \qform{-1})$.
Since $t := (n - r) - (\alpha + 2\gamma) \geqslant 0$, we can choose
$b_1 = \pm 1$ and $b_3 = (\pm 1, \ldots , \pm 1)$ so that the
quadratic form associated with $(\phi_{1 , 3})_{(b_1 , b_3)}$ is
equivalent to $(2\beta + 1 - t)\qform{ 1} \perp t \qform{-1}$. Then $b = (b_1, \ldots , b_4)$ is as required. By
Proposition \ref{P:exist}(iii), there exists an embedding $(E ,
\sigma) \hookrightarrow (A , \tau)$, and therefore an $\R$-defined
embedding $\mathrm{SU}(E , \sigma) \hookrightarrow \mathrm{SU}(A ,
\tau) = G$. Finally, it follows from our construction and Table
\ref{R.tori} that $T = \mathrm{SU}(E , \sigma)$ is a torus having
the required structure.
\end{proof}

\begin{table}[b!t]
\[
\begin{array}{cc|cc}
E&\s&\text{$\phi_b$ for $b \in E^\s$}&\SU\Es \\
\hline
\R&\Id&\qform{b}& \{ 1 \} \\
\R \times \R &\text{switch}&\qform{1,-1}& \GL_1\\
\C&\text{conjugation}&\qform{b,b}& \mathrm{R}^{(1)}_{\C/\R}(\GL_1) \\
\C \times \C&\text{switch}&\qform{1,-1} \oplus \qform{1,-1}&
\mathrm{R}_{\C/\R}(\GL_1)
\end{array}
\]
\caption{Isomorphism classes of indecomposable \'etale $\R$-algebras
with involution and their associated symmetric bilinear forms and
unitary groups.} \label{R.tori}
\end{table}

\begin{lem}[Simply connected $\tC_{\ell}$ over $\R$]  \label{C.R}
The maximal $\R$-tori in the group $G = \Sp(r , \ell - r)$ are of
the form \eqref{E:r-tori} with $\alpha = 0$, $\beta + 2\gamma =
\ell$ and $\gamma \leqslant s := \min(r , \ell - r)$.
\end{lem}

\begin{proof}
Let $\tau$ be the involution on $A = M_{\ell}(\hh)$ that gives rise
to the hermitian form $f = r\qform{ 1} \perp (\ell -
r)\qform{ - 1}$, so that $G = \mathrm{SU}(A , \tau)$. Every
maximal $\R$-torus $T$ of $G$ is of the form $T = \mathrm{SU}(E ,
\sigma)$ for some $(2\ell)$-dimensional \'etale $\tau$-invariant
subalgebra $E$ of $A$, where $\sigma = \tau \vert_E$ and condition
\eqref{E:dimen} holds. As in Example \ref{GLn.eg}, $E \simeq
\C^{\ell}$ as $\R$-algebras, and therefore
$$
(E , \sigma) = \C^{\delta_1} \times (\C \times \C)^{\delta_2}
$$
where the involutions on $\C$ and $\C \times \C$ are as in Table
\ref{R.tori}. Then in \eqref{E:r-tori} for $T = \mathrm{SU}(E ,
\sigma)$ we have $\alpha = 0$, $\beta = \delta_1$ and $\gamma =
\delta_2$. By dimension count, we get $\beta + 2\gamma = \ell$.
Furthermore,
$$
\gamma = \rk_{\R} T \leqslant \rk_{\R} G = s.
$$

Conversely, suppose that $T$ has parameters $\alpha, \beta$ and
$\gamma$ satisfying our constraints. Consider
$$
(E , \sigma) = \C^{\beta} \times (\C \times \C)^{\gamma}
$$
with the involutions as above, and assume (as we may) that $\ell - r
\leqslant r$. Note that
$$
(z , w) \mapsto \begin{pmatrix} z & 0 \\ 0 & \bar{w}
\end{pmatrix}
$$
defines an embedding of algebras with involutions $\C \times \C
\hookrightarrow (M_2(\hh) , \theta)$ where $\theta(x) =
J^{-1} \bar{x}^t J$ with $J = \left( \begin{smallmatrix} 0 & 1 \\
1 & 0 \end{smallmatrix} \right)$, where $\bar{x}$ is obtained by applying
quaternionic conjugation to all entries. Consider the involution
$\hat{\theta}$ on $A$ given by $\hat{\theta}(x) = \hat{J}^{-1}
\bar{x}^t \hat{J}$ where
$$
\hat{J} = \mathrm{diag}(\underbrace{1, \ldots , 1}_{r - \gamma} \: ,
\: \underbrace{-1, \ldots , -1}_{\beta - (r -\gamma)} \: , \:
\underbrace{J, \ldots , J}_{\gamma}).
$$
Then it follows from our construction that there exists an embedding
$(E , \sigma) \hookrightarrow (A , \theta)$. Noting that $(A , \tau)
\simeq (A , \theta)$, we obtain an embedding $(E , \sigma)
\hookrightarrow (A , \tau)$. So, there exists an $\R$-embedding
$\mathrm{SU}(E , \sigma) \hookrightarrow \mathrm{SU}(A , \tau) = G$,
and it remains to observe that $T = \mathrm{SU}(E , \sigma)$ is a
torus having the required structure.
\end{proof}

Alternatively, the results of Lemmas \ref{B.R} and \ref{C.R} can be deduced from the more general
classification of maximal $\R$-tori in simple real algebraic
groups obtained in \cite{DjThang}.  For the reader's convenience we have included the direct proofs above, written in the same language as the rest of the paper.

\begin{cor}\label{C:3.4}
Let $G_1$ be an adjoint real group of type $\tB_{\ell}$, and let
$G_2$ be a simply connected real group of type $\tC_{\ell}$. 
The groups $G_1$ and $G_2$ have the same isomorphism classes of
maximal $\R$-tori if and only if $G_1$ and $G_2$ are either both
split or both anisotropic.
\end{cor}

\begin{proof}
Since every $\R$-anisotropic torus $T$ is of the form
$(\mathrm{R}^{(1)}_{\C/\R}(\GL_1))^{\dim T}$, there is nothing to prove if
both groups are anisotropic. If both groups are split, our claim
follows from Remark \ref{R:same}. Clearly, $G_1$ and $G_2$ cannot
have the same maximal tori if one of the groups is anisotropic and the
other is isotropic. So, it remains to consider the case where both
groups are isotropic but not split. Then $G_1$ contains the torus
with $\alpha = 1$, $\beta = \ell - 1$, and $\gamma = 0$
by Lemma \ref{B.R}, but $G_2$ does not by Lemma \ref{C.R}.
\end{proof}

\begin{rmk}
Our argument shows that if $G_1$ is isotropic and $G_2$ is not split
then $G_1$ has a maximal $\R$-torus that is not isomorphic to any
$\R$-torus of $G_2$. Moreover, by Lemma \ref{B.R}, a maximal
$\R$-torus $T_1$ of $G_1$ that contains a maximal $\R$-split torus
has parameters $\alpha = s$,  $\beta = \ell - s$ and $\gamma = 0$,
hence does not allow an $\R$-embedding into $G_2$. In particular, if
$G_1 = \mathrm{SO}(n - 1 , 1)$ and $G_2$ is not split then every
isotropic maximal $\R$-torus of $G_1$ is not isomorphic to a
subtorus of $G_2$.
\end{rmk}

\begin{eg}[Absolute rank 3]\label{R.eg}
As an empirical illustration of the landscape over $\R$, we divide
the 14 real groups of types $\tB_3$ and $\tC_3$ into equivalence
classes under the relation ``have isomorphic collections of maximal
tori''. For forms of $\SO_7$ or $\Sp_6$, the maximal tori are
described by Lemmas \ref{B.R} and \ref{C.R}. Also, the four anisotropic
(compact) forms obviously make up one equivalence class. For the
other groups one can use a computer program such as the Atlas
software \cite{AdamsFokko} to find the maximal tori.  In summary, the
groups $\SO(1,6)$, $\SO(2,5)$, and $\Spin(2,5)$ are each their own
equivalence class, and we find the following non-singleton
equivalence classes:
\begin{gather*}
\{ \text{4 anisotropic forms} \}, \quad \{ \Sp_6, \SO(4,3) \}, \quad \{ \PSp_6, \Spin(4,3) \},  \\
 \text{and} \quad \{ \Sp(1,2), \PSp(1,2), \Spin(1,6) \}.
\end{gather*}
In particular, $\Spin(1,6)$ and $\PSp(1,2)$ have the same
isomorphism classes of maximal tori and yet are neither both split
nor both anisotropic.  This situation is dual to the one
considered and eliminated in Corollary \ref{C:3.4} (adjoint
$\tB_{\ell}$ and simply connected $\tC_{\ell}$).
\end{eg}

For completeness, we mention the (much easier) analogue of Corollary \ref{C:3.4} for non-archimedean local fields. 

\begin{lem} \label{nonarch}
Let $G_1$ and $G_2$ be absolutely almost simple groups of type $\tB_\ell$ and $\tC_\ell$ respectively, with $\ell \ge 3$, over $K$ a nonarchimedean local field of characteristic $\ne 2$.  The following are equivalent:
\begin{enumerate}
\item The groups $G_1$ and $G_2$ have the same isogeny classes of maximal $K$-tori.
\item $\rk_K G_1 = \rk_K G_2$.
\item $G_1$ and $G_2$ are split.
\end{enumerate}
\end{lem}

\begin{proof}
(1) obviously implies (2).  Suppose (2) and that $G_2$ is not split.  Then 
\[
[\ell/2] = \rk_K G_2 = \rk_K G_1 \ge \ell -1,
\]
but this is impossible because $\ell \ge 3$, hence (3).

To prove (3) $\Rightarrow$ (1), we may assume that $G_1$ is split adjoint and $G_2$ is split simply connected.  
Combining Propositions \ref{C.split} and \ref{B.split} with \eqref{E:decomp} gives that $G_1$ and $G_2$ have the same isogeny classes of maximal tori.
\end{proof}

\section{Local-global principles for embedding \'etale algebras with involution} \label{LG}

The last ingredient we need to develop before proving Theorem
\ref{BC} in \S\ref{S:BC} is a result guaranteeing in our
situation the validity of the local-global principle for the
existence of an embedding of an \'etale algebra with involution into
a simple algebra with involution. This issue was analyzed in
\cite{PrRap:Deven}: although the local-global principle may fail
(cf.~Example 7.5 in \emph{loc.\ cit.}), it can be shown to hold under
rather general conditions. For our purposes we need the following
case.

Let $(E , \sigma)$ be an \'etale algebra with involution over a
number field $K$ of dimension $n = 2m$ and satisfying
\eqref{E:dimen}. Then $E = F[x]/(x^2 - d)$ where $F = E^{\sigma}$ is
an $m$-dimensional \'etale $K$-algebra and $d \in F^{\times}$, with
the involution defined by $x \mapsto -x$ as in Example \ref{quad.etale}.  We write $F = \prod_{j = 1}^r F_j$ where $F_j$
is a field extension of $K$, and suppose that in terms of this
decomposition $d = (d_1, \ldots , d_r)$. Let $\tau$ be an orthogonal
involution on $A = M_n(K)$.
\begin{prop}\label{P:local-global}
{\rm \cite[Theorem 7.3]{PrRap:Deven}} Assume that for every $v \in
V^K$ there exists a $K_v$-embedding
$$
\iota_v \colon (E \otimes_K K_v , \sigma \otimes \mathrm{id}_{K_v})
\hookrightarrow (A \otimes_K K_v , \tau \otimes \mathrm{id}_{K_v}).
$$
If the following condition holds
\begin{speceqn}{$\diamond$}
\begin{equation}
 \parbox[t]{4in}{for every finite subset $V
\subset V^K$, there exists $v_0 \in V^K \setminus V$ such that for
$j = 1, \ldots , r$, if $d_j \notin {F^{\times}_j}^2$, then $d_j
\notin {(F_j \otimes_K K_{v_0})^{\times 2}}$;}
\end{equation}
\end{speceqn}
then there exists an embedding $\iota \colon (E , \sigma)
\hookrightarrow (A , \tau)$. Furthermore, $(\diamond)$ automatically
holds if $F$ is a field.
\end{prop}

We will now derive from the proposition the following statement, in
which $n$ can be odd or even.
\begin{lem}\label{LG.lem}
Let $K$ be a number field, let $(E , \sigma)$ be an $n$-dimensional
\'etale algebra with involution satisfying \eqref{E:dimen}, and let
$\tau$ be an orthogonal involution on $A = M_n(K)$. Assume that for
every $v \in V^K$ there is an embedding
$$
\iota_v \colon (E \otimes_K K_v , \sigma \otimes \mathrm{id}_{K_v})
\hookrightarrow (A \otimes_K K_v , \tau \otimes \mathrm{id}_{K_v}).
$$
Then in each of the following situations
\begin{enumerate}
\item \label{LG.deg5} $n \leqslant 5$, or
\item \label{LG.real} there is a real $v \in V^K$ such that $(E
\otimes_K K_v , \sigma \otimes \mathrm{id}_{K_v})$ is isomorphic to
$(\C , \bar{\ })^m$ or $(\C , \bar{\ })^m \times (\R ,
\mathrm{id}_{\R})$ depending on whether $n = 2m$ or $n = 2m + 1$,
\end{enumerate}
there exists an embedding $\iota \colon (E , \sigma)
\hookrightarrow (A , \tau)$.
\end{lem}

\begin{proof}
First, we will reduce the argument to the case of even $n$, i.e.
when $E$ satisfies one of the following conditions:
\begin{enumerate}
\renewcommand{\theenumi}{\arabic{enumi}$'$}
\item \label{1p} $n = 2$ or $4$, or

\item \label{2p} there is a real $v \in V^K$ such that
$(E \otimes_K K_v , \sigma \otimes \mathrm{id}_{K_v})$ is isomorphic
to $(\C , \bar{\ })^m$.
\end{enumerate}
Indeed, let $n = 2m + 1$ and suppose $E$ satisfies condition (1)
or (2) of the lemma. Then by \cite[Prop. 7.2]{PrRap:Deven}, $(E ,
\sigma) = (E' , \sigma') \times (K , \mathrm{id}_K)$ and there
exists an orthogonal involution $\tau'$ on $A' = M_{n-1}(K)$ such
that for every $v \in V^K$ there is an embedding
$$
\iota'_v \colon (E' \otimes_K K_v ,  \sigma' \otimes
\mathrm{id}_{K_v}) \hookrightarrow (A' \otimes_K K_v , \tau' \otimes
\mathrm{id}_{K_v}),
$$
and the existence of an embedding $\iota' \colon (E' , \sigma')
\hookrightarrow (A' , \tau')$ is equivalent to the existence of an
embedding $\iota \colon (E , \sigma) \hookrightarrow (A , \tau)$.
Clearly, $E'$ satisfies the respective condition $\eqref{1p}$ or $\eqref{2p}$.
So, if we assume that the lemma has already been established for
$E'$, then the existence of $\iota$ follows.

Now, suppose that $\dim_K E = 2m$ and $E$ satisfies \eqref{E:dimen}.
Write $E = F[x]/(x^2 - d)$ where $F = E^{\sigma} = \prod_{j = 1}^r
F_j$, $d = (d_1, \ldots , d_r)$ with $d_j \in F^{\times}_j$. Assume
that there exist  $K$-embeddings $\varphi_j \colon F_j
\hookrightarrow \bar{K}$ such that if 
\[
M = \varphi_1(F_1) \cdots
\varphi_r(F_r) \quad \text{and} \quad  N = M(\sqrt{\varphi_1(d_1)},
\ldots , \sqrt{\varphi_r(d_r)})
\]
then there is $\lambda \in
\mathrm{Gal}(N/M)$  with the property
\begin{equation}\label{E:Frob}
\lambda\left( \sqrt{\varphi_j(d_j)} \right) = -\sqrt{\varphi_j(d_j)}
\quad \text{whenever $d_j \notin F^{\times}_j$ for $j = 1, \ldots , r$.}
\end{equation}
Let $P$ be the normal closure of $N$ over $K$, and let $\mu \in
\mathrm{Gal}(P/K)$ be such that $\mu \vert_N = \lambda$. By
Chebotarev's Density Theorem \cite[Ch.~7, 2.4]{CasselsFroh}, for any finite $V
\subset V^K$, there exists a nonarchimedean $v_0 \in V^K \setminus
V$ that is unramified in $P$ and for which the Frobenius
automorphism $\mathrm{Fr}(w_0 | {v_0})$ is $\mu$ for a suitable
extension $w_0 \vert v_0$. Then it follows from \eqref{E:Frob} that
$d_j \notin {({F_j}_{w_0})^{\times 2}}$ for any $j$ such that $d_j
\notin {F^{\times 2}_j}$, and therefore condition $(\diamond)$
holds.

Let now $(E , \sigma)$ be an \'etale algebra with involution
satisfying \eqref{1p} or \eqref{2p} for which embeddings $\iota_v$ exist for
all $v \in V^K$. In order to derive the existence of $\iota$ from
Proposition \ref{P:local-global}, we need to check $(\diamond)$, for
which it is enough to find an automorphism $\lambda$ as in the
previous paragraph. Suppose that \eqref{1p} hods. Then $F = E^{\sigma}$
has dimension 1 or 2 respectively. Since we don't need to consider
the case where $F$ is a field (cf.\ Proposition
\ref{P:local-global}), the only remaining case is where $F = K
\times K$. Clearly, $K(\sqrt{d_1} , \sqrt{d_2})$ always has an
automorphism $\lambda$ such that $\lambda(\sqrt{d_j}) = -
\sqrt{d_j}$ if $d_j \notin {K^{\times 2}}$, as required. Finally,
suppose that \eqref{2p} holds. Then $F \otimes_K K_v \simeq \R^m$, and
$d = (\delta_1, \ldots , \delta_m)$ in $\R^m$ with $\delta_i < 0$
for all $i$. Then for any embeddings $\varphi_j \colon F_j
\hookrightarrow \C$ we have $\varphi_j(F_j) \subset \R$ and the
restriction $\lambda$ of complex conjugation satisfies
$\lambda(\sqrt{d_j}) = - \sqrt{d_j}$ for all $j$, concluding the
argument.
\end{proof}

\begin{rmk*}
Example 7.5 in \cite{PrRap:Deven} shows that there exists $(E ,
\sigma)$ with $E$ of dimension 6 for which the local-global
principle for embeddings fails, so in terms of dimension the
condition  $(1)$ in Lemma \ref{LG.lem}  is sharp.
\end{rmk*}

For convenience of further reference, we will also quote the
local-global principle for embeddings in the case of symplectic
involutions.
\begin{lem}\label{L:sympl}
{\rm \cite[Th.~5.1]{PrRap:Deven}} Let $A$ be a central
simple $K$-algebra of dimension $n^2$ with a symplectic involution
$\tau$ (then, of course, $n$ is necessarily even), and let $(E ,
\sigma)$ be an $n$-dimensional \'etale $K$-algbra with involution
satisfying \eqref{E:dimen}. If for every $v \in V^K$ there exists an
embedding
$$
\iota_v \colon (E \otimes_K K_v , \sigma \otimes \mathrm{id}_{K_v})
\hookrightarrow (A \otimes_K K_v , \tau \otimes \mathrm{id}_{K_v}),
$$
then there exists an embedding $(E , \sigma) \hookrightarrow (A ,
\tau)$.
\end{lem}

\section{Function field analogue of Theorem \ref{BC}} \label{ffield}

We recall the following immediate consequence of the
rationality of the variety of maximal tori (see \cite{Hrdr:remark} or \cite[Cor.~7.3]{PlatRap}) which will be used repeatedly: \emph{Let $G$ be a
reductive algebraic group over a number field $K$; then given any $v
\in V^K$ and any maximal $K_v$-torus $T^{(v)}$ of $G$ there exists a
maximal $K$-torus $T$ of $G$ that is conjugate to $T^{(v)}$ by an
element of $G(K_v)$.} In particular, for any $v \in V^K$ there
exists a maximal $K$-torus $T$ of $G$ such that $\rk_{K_v} 
T = \rk_{K_v}  G$. It follows that if $G_1$ and $G_2$ are
reductive $K$-groups having the same isogeny classes of maximal
$K$-tori then
\begin{equation}\label{E:equal-rank}
\rk_{K_v}  G_1 = \rk_{K_v}  G_2
\quad \text{for all $v \in V^K$.}
\end{equation}

The remark made in the previous paragraph remains valid for global function fields, which can be used to give the following analogue of Theorem \ref{BC}:
Suppose $G_1$ and $G_2$ are 
absolutely almost simple algebraic groups of types $\tB_{\ell}$ and
$\tC_{\ell}$ $(\ell \geqslant 3)$ over a global field $K$ of characteristic $> 2$.  \emph{The groups $G_1$ and $G_2$ have the
same isogeny classes of maximal $K$-tori if and only if they are split.}  Indeed, if the two groups have the same isogeny classes of maximal $K$-tori, then both groups are $K_v$-split for every $v$ (by \eqref{E:equal-rank} and Lemma \ref{nonarch}), hence  both groups are $K$-split (by the Hasse Principle).  The converse holds by Remark \ref{R:same}.

\section{Proof of Theorem \ref{BC}}\label{S:BC}

Throughout this section $G_1$ and $G_2$ will denote absolutely
almost simple algebraic groups of types $\tB_{\ell}$ and
$\tC_{\ell}$ for some $\ell \geqslant 3$ defined over a number field
$K$. 
In \ref{twin.def} we defined what it means for $G_1$
and $G_2$ to be \emph{twins}. We now observe that since $G_1$ and
$G_2$ cannot be $K_v$-anisotropic for $v \in V^K_f$, they are twins if
and only if both of the following conditions hold:
\begin{eqnarray}
\text{$\rk_{K_v} G_1 = \rk_{K_v} G_2 = \ell$ for
all $v \in V^K_f$} \label{I} \\
\text{$\rk_{K_v} G_1 = \rk_{K_v} G_2 = 0$ or
$\ell$ for all $v \in V^K_{\infty}$.} \label{II}
\end{eqnarray}
We also note that if $G_1$ and $G_2$ are twins over $K$
then they remain twins over any finite extension $L/K$.  If $K$ has $r$ real places, then (by the Hasse Principle) there are exactly $4 \cdot 2^r$ pairs of $K$-groups $G_1$, $G_2$ that are twins, equivalently, $2^r$ pairs if one only counts the groups $G_1$ and $G_2$ up to isogeny.

\medskip

Now, let $G_1$ and $G_2$ be as above, with $G_1$ adjoint and $G_2$
simply connected. Then $G_i = \mathrm{SU}(A_i , \tau_i)$ for $i = 1,
2$ where $A_1 = M_{n_1}(K)$, $n_1 = 2\ell + 1$ and the involution
$\tau_1$ is orthogonal, and $A_2$ is a central simple $K$-algebra of
dimension $n^2_2$ with $n_2 = 2\ell$ and the involution $\tau_2$ is
symplectic. Any maximal $K$-torus $T_i$ of $G_i$ is of the form
$\mathrm{SU}(E_i , \sigma_i)$ where $E_i \subset A_i$ is an
$n_i$-dimensional \'etale $\tau_i$-invariant $K$-subalgebra and
$\sigma_i = \tau_i \vert_{E_i}$ so that \eqref{E:dimen} holds. For $i
= 1$, we can always write $(E_1 , \sigma_1) = (E'_1 , \sigma'_1)
\times (K , \mathrm{id}_K)$. For $i = 2$, we set $(E^+_2 ,
\sigma^+_2) = (E_2 , \sigma_2) \times (K , \mathrm{id}_K)$.
\begin{prop}\label{P:same-subalg}
Let $(A_1 , \tau_1)$ and $(A_2 , \tau_2)$ be algebras with
involution as above, and assume that $G_1 = \mathrm{SU}(A_1 ,
\tau_1)$ and $G_2 = \mathrm{SU}(A_2 , \tau_2)$ are twins. If $(E_1 ,
\sigma_1)$ is isomorphic to an $n_1$-dimensional \'etale subalgebra
of $(A_1 , \tau_1)$ satisfying \eqref{E:dimen} then $(E'_1 ,
\sigma'_1)$ is isomorphic to a subalgebra of $(A_2 , \tau_2)$.
Conversely, if $(E_2 , \sigma_2)$ is isomorphic to an
$n_2$-dimensional \'etale subalgebra of $(A_2 , \tau_2)$ satisfying
\eqref{E:dimen} then $(E^+_2 , \sigma^+_2)$ is isomorphic to a
subalgebra of $(A_1 , \tau_1)$. Thus, the correspondences 
\[
(E_1 ,
\sigma_1) \mapsto (E'_1 , \sigma'_1) \quad \text{and} \quad (E_2 ,
\sigma_2) \mapsto (E^+_2 , \sigma^+_2)
\]
 implement mutually inverse
bijections between the sets of isomorphism classes of $n_1$- and
$n_2$-dimensional \'etale subalgebras of
$(A_1 , \tau_1)$ and $(A_2 , \tau_2)$ that are invariant under the respective involutions and satisfy
\eqref{E:dimen}.
\end{prop}

\begin{proof}
If we have $\rk_{K_v} G_1 =
\rk_{K_v} G_2 = \ell$ for all $v \in V^K_{\infty}$ then
the groups $G_1$ and $G_2$ are $K$-split by \eqref{I} and the Hasse Principle. Then $\tau_1$
corresponds to a nondegenerate symmetric bilinear form of Witt index
$\ell$, and $A_2 = M_{n_2}(K)$ with $\tau_2$ corresponding to a
nondegenerate skew-symmetric form. In this case, our claim
immediately follows from Propositions \ref{C.split} and
\ref{B.split}, as in Remark \ref{R:same}. So, we may assume that
there is a real $v_0 \in V^K_{\infty}$ such that
$\rk_{K_{v_0}}  G_1 = \rk_{K_{v_0}} G_2 = 0$.
Observe that given \emph{any} real $v \in V^K_{\infty}$ satisfying
$\rk_{K_v}  G_1 = \rk_{K_v}  G_2 = 0$, the data in
Table \ref{R.tori} shows that that for any $n_1$-dimensional
$\tau_1$-invariant \'etale subalgebra $E_1 \subset A_1$ satisfying
\eqref{E:dimen} and $\sigma_1 = \tau_1 \vert_{E_1}$ we have
\begin{equation}\label{E:E1}
(E_1 \otimes_K K_v , \sigma_1 \otimes \mathrm{id}_{K_v}) \simeq (\C
, \bar{\ })^{\ell} \times (\R , \mathrm{id}_{\R}),
\end{equation}
and for any $n_2$-dimensional $\tau_2$-invariant \'etale subalgebra
$E_2 \subset A_2$ satisfying \eqref{E:dimen} and $\sigma_2 = \tau_2
\vert_{E_2}$ we have
\begin{equation}\label{E:E2}
(E_2 \otimes_K K_v , \sigma_2 \otimes \mathrm{id}_{K_v}) \simeq (\C
, \bar{\ })^{\ell}.
\end{equation}
Let $(E_1 , \sigma_1)$ be as in the statement of the proposition. We
first show that for any $v \in V^K$ there is an embedding
$$
\iota_v \colon (E'_1 \otimes_K K_v , \sigma'_1 \otimes
\mathrm{id}_{K_v}) \hookrightarrow (A_2 \otimes_K K_v , \tau_2
\otimes \mathrm{id}_{K_v}).
$$
If $\rk_{K_v}  G_1 = \rk_{K_v}  G_2 = \ell$, this
follows from Proposition \ref{C.split}. Otherwise, $v$ is real, and
$\rk_{K_v}  G_1 = \rk_{K_v}  G_2 = 0$, so we see
from \eqref{E:E1} that
$$
(E'_1 \otimes_K K_v , \sigma'_1 \otimes \mathrm{id}_{K_v}) \simeq
(\C , \bar{\ })^\ell.
$$
Then the existence of $\iota_v$ follows from the argument given in
the proof of Lemma \ref{C.R}. Now, applying Lemma \ref{L:sympl} we
obtain the existence of an embedding $\iota \colon (E'_1 ,
\sigma'_1) \hookrightarrow (A_2 , \tau_2)$, as required.

Conversely, let $(E_2 , \sigma_2)$ be as in the proposition. Then
arguing as above (using Proposition \ref{B.split} and the proof of
Lemma \ref{B.R}) we obtain the existence of local embeddings
$$
\iota_v \colon (E^+_2 \otimes_K K_v , \sigma^+_2 \otimes
\mathrm{id}_{K_v}) \hookrightarrow (A_1 \otimes_K K_v , \tau_1
\otimes \mathrm{id}_{K_v})
$$
for all $v \in V^K$. It follows from \eqref{E:E2} that
$$
(E^+_2 \otimes_K K_{v_0} , \sigma^+_2 \otimes \mathrm{id}_{K_{v_0}})
\simeq (\C , \bar{\ })^{\ell} \times (\R , \mathrm{id}_{\R}).
$$
This enables us to use Lemma \ref{LG.lem} which yields the existence
of an embedding $(E^+_2 , \sigma^+_2) \hookrightarrow (A_1 ,
\tau_1)$, completing the argument.
\end{proof}

\medskip

The following consequence of the proposition proves the ``if''
component in both parts, (1) and (2), of Theorem \ref{BC}.

\begin{cor}\label{C:BC-twins}
Let $G_1$ and $G_2$ be absolutely almost simple algebraic groups of
types $\tB_{\ell}$ and $\tC_{\ell}$ respectively that are
twins. Then
\begin{enumerate}
\renewcommand{\theenumi}{\roman{enumi}}
\item  \label{twin.isog}
$G_1$ and $G_2$
have the same \underline{isogeny} classes of maximal $K$-tori.
\item \label{twin.same} If $G_1$ is adjoint and $G_2$ is
simply connected then $G_1$ and $G_2$ have the same
\underline{isomorphism} classes of maximal $K$-tori.
\end{enumerate}
\end{cor}

\begin{proof}
\eqref{twin.same} easily follows from the proposition, and \eqref{twin.isog} is an immediate
consequence of \eqref{twin.same}.
\end{proof}

\begin{rmk}\label{R:L=2}
The assumption $\ell \geqslant 3$ was never used in Proposition
\ref{P:same-subalg} and Corollary \ref{C:BC-twins}. So, these
statements remain valid also for $\ell = 2$, which will be helpful
in \S~\ref{S:Compl}.
\end{rmk}

We now turn to the proof of the ``only if'' direction in both parts
of Theorem \ref{BC} where the assumption $\ell \geqslant 3$ becomes
essential and will be kept throughout the rest of the section.
This direction requires a bit more work and
involves the notion of \emph{generic tori}. To recall the relevant
definitions, we let $G$ denote a semi-simple algebraic $K$-group,
and  fix a maximal $K$-torus $T$ of $G$. Furthermore, we let $\Phi(G
, T)$ denote the corresponding root system, and let $K_T$ denote the
minimal splitting field of $T$ over $K$. The natural action of
$\mathrm{Gal}(K_T/K)$ on the group of characters $X(T)$ gives rise
to an injective group homomorphism
$$
\theta_T \colon \mathrm{Gal}(K_T/K) \longrightarrow
\mathrm{Aut}(\Phi(G , T)).
$$
Then $T$ is called \emph{generic} (over $K$) if
$\theta_T(\mathrm{Gal}(K_T/K))$ contains the Weyl group $W(G , T)$.
As the following statement shows, generic tori with prescribed local
properties always exist.
\begin{prop}\label{P:generic-tori}
{\rm \cite[Corollary 3.2]{PrRap:weakly}} Let $G$ be an absolutely
almost simple algebraic $K$-group, and let $V \subset V^K$ be a
finite subset. Suppose that for each $v \in V$ we are given a
maximal $K_v$-torus $T^{(v)}$ of $G$. Then there exists a maximal
$K$-torus $T$ of $G$ which is generic over $K$ and which is
conjugate to $T^{(v)}$ by an element of $G(K_v)$ for all $v \in V$.
\end{prop}

We now return to the situation where $G_1$ and $G_2$ are absolutely
almost simple $K$-groups of types $\tB_{\ell}$ and $\tC_{\ell}$
$(\ell \geqslant 3)$ respectively. We let $\nat{G}_1$ denote the
adjoint group of $G_1$, and  $\nat{G}_2$ the simply connected cover
of $G_2$. Furthermore, given a maximal $K$-torus $T_i$ of $G_i$, we
let $\nat{T}_i$ denote the image of $T_i$ in $\nat{G}_i$ if $i = 1$
and the preimage of $T_i$ in $\nat{G}_i$ if $i = 2$.

\begin{prop}\label{P:isom}
Let $T_i$ be a generic maximal $K$-torus of $G_i$ where $i = 1, 2$.
If there exists a $K$-isogeny $\pi \colon T_i \to T_{3-i}$ onto a
maximal $K$-torus of $G_{3-i}$ then there exist a $K$-isomorphism
$\nat{T}_i \simeq \nat{T}_{3-i}$.
\end{prop}

The proof below is an adaptation of
Lemma 4.3 and Remark 4.4 in \cite{PrRap:weakly}.

\begin{proof}
We have $K_{T_1} = K_{T_2} =: L$, and let $\mathcal{G} =
\mathrm{Gal}(L/K)$. Then $\theta_{T_j}$ is an isomorphism of
$\mathcal{G}$ on $W_j = W(G_j , T_j)$ for $j = 1, 2$. The isogeny
$\pi$ induces a $\mathcal{G}$-equivariant homomorphism of character
groups $\pi^* \colon X(T_{3-i}) \to X(T_i)$. Let $\nat{X}_j =
X(\nat{T}_j)$; we need to prove that there is a
$\mathcal{G}$-equivariant \emph{isomorphism} $\psi \colon
\nat{X}_{3-i} \to \nat{X}_i$. (We recall that $\nat{X}_1$ is the
subgroup of $X(T_1)$ generated by all the roots in $\Phi_1 =
\Phi(G_1 , T_1)$, and $\nat{X}_2$ is generated by the weights of the
root system $\Phi_2 = \Phi(G_2 , T_2)$.)

To avoid cumbersome notations, we will assume that $i = 1$.  (This
does not restrict generality as along with $\pi$ there is always a
$K$-isogeny $\pi' \colon T_{3-i} \to T_i$.) Consider
$$
\phi = \pi^* \otimes \mathrm{id}_{\R} \colon V_2 = X(T_2)
\otimes_{\Z} \R \longrightarrow X(T_1) \otimes_{\Z} \R = V_2
$$
and $\mu \colon W_2 \to W_1$ defined by $\mu = \theta_{T_1} \circ
\theta^{-1}_{T_2}$. Then the fact that $\pi^*$ is
$\mathcal{G}$-equivariant implies that
\begin{equation}\label{E:equiv1}
\phi(w \cdot v) = \mu(w) \cdot \phi(v) \ \ \text{for all} \ \  v \in
V_2, \ w \in W_2.
\end{equation}
On the other hand, it follows from the explicit description of the
root systems as in \cite{Bou:g4} that there exists a linear
isomorphism $\phi_0 \colon V_2 \to V_1$ and a group isomorphism
$\mu_0 \colon W_2 \to W_1$ such that
\begin{equation}\label{E:equiv2}
\phi_0(w \cdot v) = \mu_0(w) \cdot \phi_0(v) \ \ \text{for all} \ \
v \in V_2, \ w \in W_2,
\end{equation}
$\phi_0$ takes the short roots of $\Phi_2$ to the long roots of
$\Phi_1$, and $(1/2)\phi_0$ takes the long roots of $\Phi_2$ to the
short roots of $\Phi_1$, consequently $\phi_0(\nat{X}_2) =
\nat{X}_1$. (Note that we identify $W_j$ with the Weyl group of the
root system $\Phi_j$.) 

We claim that there exists a nonzero $\lambda \in \R$ and $z \in W_1$ such that
$$
\phi(v) = \lambda \cdot z \cdot \phi_0(v) \ \ \text{and} \ \ \mu(w)
= z \cdot \mu_0(w) \cdot z^{-1} \quad \text{for all $v \in V_2$, $w \in W_2$.}
$$
Indeed, it was shown in \cite[Lemma 4.3]{PrRap:weakly} (using that $\ell \geqslant 3$) that a suitable multiple $\phi' =
\lambda^{-1} \cdot \phi$ takes the short roots of $\Phi_2$ to the
long roots of $\Phi_2$, and $(1/2)\phi_0$ takes the long roots of
$\Phi_2$ to the short roots of $\Phi_1$. Then $z := \phi' \circ
\phi^{-1}_0$ is an automorphism of $\Phi_1$, hence can be identified
with an element of $W_1$. This gives the formula for $\phi$, and
then the formula for $\mu$ follows from \eqref{E:equiv1} and
\eqref{E:equiv2}.

Put $\psi := \lambda^{-1} \cdot \phi$.
Then
$$
\psi(\nat{X}_2) = z(\phi_0(\nat{X}_2)) = \nat{X}_1,
$$
and $\psi$ is $\mathcal{G}$-equivariant, as required.
\end{proof}

\begin{cor} \label{II.cor}
Let $T_i$ be a generic maximal $K$-torus of $G_i$. If there exists
$v \in V$ such that $\nat{T}_i$ does not allow a $K_v$-defined
embedding into $\nat{G}_{3-i}$ then $T_i$ is not $K$-isogenous to
any maximal $K$-torus $T_{3-i}$ of $G_{3-i}$. Thus, if $G_1$ and
$G_2$ have the same isogeny classes of maximal $K$-tori then
$\nat{G}_1$ and $\nat{G}_2$ have the same isomorphism classes of
maximal $K_v$-tori for all $v \in V$.
\end{cor}

\begin{proof}
The first assertion immediately follows from the
proposition. To derive the second assertion from the first, we
observe that given $v \in V$ and a maximal $K_v$-torus
$\mathcal{T}_i$ of $\nat{G}_i$ that does not allow a $K_v$-embedding
into $\nat{G}_{3-i}$, we can find a maximal $K$-torus $T_i$ of $G_i$
such that $\nat{T}_i$ is conjugate to $\mathcal{T}_i$ by an element
$\nat{G}_i(K_v)$.
\end{proof}
 
\begin{proof}[Proof of Theorem \ref{BC}, ``only if'']
Assume $G_1$ and $G_2$ have the same isogeny classes of maximal $K$-tori.  Then by Corollary \ref{II.cor}, $\nat{G_1}$ and $\nat{G_2}$ have the same isomorphism classes of maximal $K_v$-tori for all $v$.  It follows that $G_1$ and $G_2$ are twins (by Corollary \ref{C:3.4} for $v$ real and Lemma \ref{nonarch} for $v$ finite), completing the proof of 
 part (1) of Theorem \ref{BC}.

Now suppose that $G_1$ and $G_2$ have the same \emph{isomorphism} classes of
maximal $K$-tori, in particular,  there is a $K$-isomorphism $\pi
\colon T_1 \to T_2$ between two generic $K$-tori. Then as in the proof of Proposition \ref{P:isom},
$\pi^*$ induces $\phi \colon V_2 \to V_1$ which necessarily
satisfies $\phi(X(T_2)) = X(T_1)$ and $\phi(X(\nat{T}_2)) =
X(\nat{T}_1)$. Since $X(\nat{T}_1) \subseteq X(T_1)$ and $X(\nat{T}_2)
\supseteq X(T_2)$, this is possible only if both inclusions are in
fact equalities, i.e., $G_1 = \nat{G}_1$ and $G_2 = \nat{G}_2$.  This completes the proof of part (2) of Theorem \ref{BC}.
\end{proof}

\section{Weakly commensurable subgroups and proof of Theorem
\ref{T:1}}\label{S:WC}

We begin by recalling the notion of weak commensurability of
Zariski-dense subgroups introduced in \cite{PrRap:weakly}. Let $G_1$
and $G_2$ be semi-simple algebraic groups over a field $F$ of
characteristic zero, and let $\Gamma_i \subset G_i(F)$ be a
Zariski-dense subgroup for $i = 1, 2$. Semi-simple elements
$\gamma_i \in \Gamma_i$ are \emph{weakly commensurable} if there
exist maximal $F$-tori $T_i$ of $G_i$ such that $\gamma_i \in
T_i(F)$ and for some characters $\chi_i \in X(T_i)$ we have
$$
\chi_1(\gamma_1) = \chi_2(\gamma_2) \neq 1.
$$
Furthermore, the subgroups $\Gamma_1$ and $\Gamma_2$ are
\emph{weakly commensurable} if every semi-simple element $\gamma_1
\in \Gamma_1$ of infinite order is weakly commensurable to some
$\gamma_2 \in \Gamma_2$ of infinite order, and vice versa. 

The focus in \cite{PrRap:weakly} was on analyzing when two Zariski-dense
$S$-arithmetic subgroups in absolutely almost simple algebraic
groups are weakly commensurable. This analysis was based on a
description of such $S$-arithmetic groups in terms of triples, which
we will now briefly recall. Let $G$ be a (connected) absolutely almost
simple algebraic group defined over a field $F$ of characteristic
zero, $\overline{G}$ be its adjoint group, and $\pi \colon G \to
\overline{G}$ be the natural isogeny. Suppose we are given the
following data:
\begin{itemize}
\item  a number field $K$ with a fixed embedding $K
\hookrightarrow F$;

\item a finite set $S$ of valuations of $K$ containing all
archimedean valuations; and
\item an $F/K$-form $\mathscr{G}$ of
$\overline{G}$ (i.e., a $K$-defined algebraic group such that there
exists an $F$-defined isomorphism of algebraic groups
${}_F\mathscr{G} \simeq \overline{G}$, where ${}_F\mathscr{G}$ is
the group obtained from $\mathscr{G}$ by the extension of scalars
$F/K$).
\end{itemize}
(Note that it is assumed in addition that $S$ does not
contain any nonarchimedean valuations $v$ such that $\mathscr{G}$ is
$K_v$-anisotropic.) We then have an embedding $\iota \colon
\mathscr{G}(K) \hookrightarrow \overline{G}(F)$ and a natural
$S$-arithmetic subgroup $\mathscr{G}(\mathcal{O}_K(S))$, where
$\mathcal{O}_K(S)$ is the ring of $S$-integers in $K$, defined in
terms of a fixed $K$-embedding $\mathscr{G} \hookrightarrow
\mathrm{GL}_n$, i.e. $\mathscr{G}(\mathcal{O}_K(S)) = \mathscr{G}(K)
\cap \mathrm{GL}_n(\mathcal{O}_K(S))$. A subgroup $\Gamma$ of $G(F)$
such that $\pi(\Gamma)$ is commensurable with
$\iota(\mathscr{G}(\mathcal{O}_K(S)))$ is called $(\mathscr{G}, K,
S)$-{\it arithmetic}. (It should be pointed out that we do not fix
an $F$-defined isomorphism ${}_F\mathscr{G} \simeq \overline{G}$ in
this definition, and by varying it we obtain a class of subgroups
invariant under $F$-defined automorphisms of $\mathscr{G}$ in the
obvious sense.)

It was
shown in \cite{PrRap:weakly} that if $G_i$ is absolutely almost
simple and $\Gamma_i$ is Zariski-dense and $(\G_i, K_i,
S_i)$-arithmetic for $i = 1, 2$ then the weak commensurability of
$\Gamma_1$ and $\Gamma_2$ implies that $K_1 = K_2=:K$ and $S_1 =
S_2=:S$, and additionally either $G_1$ and $G_2$ are
of the same type or one of them is of type $\tB_{\ell}$ and
the other is of type $\tC_{\ell}$ for some $\ell \ge
3$. That paper also contains many precise conditions for two
$S$-arithmetic subgroups to be weakly commensurable in the case
where $G_1$ and $G_2$ are of the \emph{same} type. The goal of this
section is to prove Theorem \ref{T:1} which provides such conditions when
one of the groups is of type $\tB_{\ell}$ and the other of
type $\tC_{\ell}$ $(\ell \ge 3)$. In conjunction with
the previous results, this completes the investigation of weak
commensurability of $S$-arithmetic subgroups in absolutely almost
simple groups over number fields.

\begin{proof}[Proof of Theorem \ref{T:1}]
Let $G_1$ and $G_2$ be absolutely almost simple algebraic groups
of types $\tB_{\ell}$ and $\tC_{\ell}$ $(\ell \geqslant 3)$
respectively defined over a number field $K$, and let $\Gamma_i$ be
a Zariski-dense $(\G_i , K, S)$-arithmetic subgroup of $G_i$. 

Suppose that $\G_1$ and $\G_2$ are twins. Then by Theorem \ref{BC},
they have the same isogeny classes of maximal $K$-tori. This
\emph{automatically} implies that $\Gamma_1$ and $\Gamma_2$ are
weakly commensurable. To see this, we basically need to repeat the
argument given in \cite[Example 6.5]{PrRap:weakly}, which we also
give here for the reader's convenience. First, we may assume without
any loss of generality that $G_1$ and $G_2$ are adjoint (cf.\ Lemma
2.4 in \cite{PrRap:weakly}), hence $\Gamma_i \subset \G_i(K)$. Let
$\gamma_1 \in \Gamma_1$ be a semi-simple element of infinite order,
and let $T_1$ be a maximal $K$-torus of $\G_1$ that contains
$\gamma_1$. Then there exists a $K$-isogeny $\varphi \colon T_1 \to
T_2$ onto a maximal $K$-torus $T_2$ of $\G_2$. The subgroup
$\varphi(T_1(K) \cap \Gamma_1)$ is an $S$-arithmetic subgroup of
$T_2(K)$, so there exists $n > 0$ such that $\gamma_2 :=
\varphi(\gamma_1)^n \in \Gamma_2$. Let $\chi_1 \in
\varphi^*(X(T_2))$ be a character such that $\chi_1(\gamma_1)$ is
not a root of unity, and let $\chi_2 \in X(T_2)$ be such that
$\varphi^*(\chi_2) = \chi_1$. Then
$$
(n\chi_1)(\gamma_1) = \chi_1(\gamma_1)^n = \chi_2(\gamma_2) \neq 1,
$$
which implies that $\Gamma_1$ and $\Gamma_2$ are weakly
commensurable.

Conversely, suppose that $\Gamma_1$ and $\Gamma_2$ are weakly
commensurable. According to \cite[Theorem 6.2]{PrRap:weakly}, this
in particular implies that
$$
\rk_{K_v}  \G_1 = \rk_{K_v}  \G_2 \quad \text{for all $v \in V^K$.}
$$
 As we have seen in Lemma \ref{nonarch}, for $v \in
V^K_f$ and the groups under consideration, the equality of ranks
implies that both groups are actually $K_v$-split, verifying
condition \eqref{I} in \S~\ref{S:BC}. Assume that condition \eqref{II} fails for
a real $v_0 \in V^K_{\infty}$. Then by Corollary \ref{C:3.4}, there
is an $i \in \{1 , 2\}$ and a maximal $K_{v_0}$-torus
$\mathcal{T}_i$ of $\nat{\G}_i$ that does not allow a
$K_{v_0}$-embedding into $\nat{\G}_{3-i}$; obviously $\mathcal{T}_i$
is $K_{v_0}$-isotropic. Let $T^{(v_0)}_i$ be a maximal
$K_{v_0}$-torus of $\G_i$ such that $(\nat{T}_i)^{(v_0)} =
\mathcal{T}_i$. Furthermore, for $v \in S \setminus \{ v_0 \}$ we
let $T^{(v)}_i$ denote a maximal $K_v$-torus of $\G_i$ such that
$\rk_{K_v}  T^{(v)}_i = \rk_{K_v}  \G_i$. Using
Proposition \ref{P:generic-tori}, we can find a maximal $K$-torus
$T_i$ of $\G_i$ that is generic and that is conjugate to $T^{(v)}$
by an element of $\G_i(K_v)$ for all $v \in S \cup \{ v_0 \}$. Then
clearly
$$
\rk_S   T_i  := \sum_{v \in S} \rk_{K_v}   T_i > 0
$$
as $\rk_S   \G_i > 0$. By Dirichlet's Theorem \cite[Theorem
5.12]{PlatRap}, the group of $S$-integral points
$T_i(\mathcal{O}_K(S))$ has the following structure: $H \times \Z^d$
where $d = \rk_S   T_i - \rk_K   T_i$. Since $T_i$
is obviously $K$-anisotropic, we conclude that there exists
$\gamma_i \in T_i(K) \cap \Gamma_i$ of infinite order (as in the
previous paragraph, we are assuming that $G_1$ and $G_2$ are
adjoint, hence $\Gamma_j \subset \G_j(K)$ for $j = 1, 2$). Then
$\gamma_i$ is weakly commensurable to some semi-simple $\gamma_{3-i}
\in \Gamma_{3-i}$ of infinite order. Let $T_{3-i}$ be a maximal
$K$-torus of $\G_{3-i}$ containing $\gamma_{3-i}$. By the Isogeny
Theorem \cite[Theorem 4.2]{PrRap:weakly}, the tori $T_i$ and
$T_{3-i}$ are $K$-isogenous. Using Proposition \ref{P:isom}, we
conclude that $\nat{T}_i$ and $\nat{T}_{3-i}$ are $K$-isomorphic.
This implies that over $K_{v_0}$, the torus $\mathcal{T}_i \simeq
\nat{T}_i$ has an embedding into $\G_{3-i}$. A contradiction,
proving \eqref{II}, and completing the proof of Theorem \ref{T:1}.
\end{proof}

As we already mentioned, the notion of weak commensurability was
introduced in order to tackle some differential-geometric problems
dealing with length-com\-men\-sur\-able and isospectral locally symmetric
spaces, and we would like to conclude this section with a sample of
geometric consequences of the results of the current paper
established in \cite{PrRap:fields}. For a Riemannian manifold $M$,
we let $L(M)$ denote the weak length spectrum of $M$, i.e., the
collection of lengths of all closed geodesics in $M$. Two Riemannian
manifolds $M_1$ and $M_2$ are called \emph{length-commensurable} if
$\Q \cdot L(M_1) = \Q \cdot L(M_2)$.
\begin{equation} \label{geom1}
\begin{tabular}{||p{4.3in}}
\emph{Let $M_1$ be an arithmetic
quotient of the real hyperbolic space ${\mathbb H}^p$ $(p \ge 5)$, and $M_2$
be an arithmetic quotient of the quaternionic hyperbolic space
${\mathbb H}^q_{\mathbf{H}}$ $(q \ge 2)$. Then $M_1$ and
$M_2$ are not length-commensurable.}
\end{tabular}
\end{equation}
Theorem \ref{T:1} is used to handle the case $p = 2n$ and $q = n-1$ for $n \ge 3$; for other values of $p$ and $q$, the claim follows from \cite[Th.~8.15]{PrRap:weakly}.

Now, let $\X_1$ be the symmetric space of the real Lie
group $\mathcal{G}_1 = \SO(n + 1 , n)$, and let $\X_2$ be
the symmetric space of the real Lie group $\mathcal{G}_2 = \Sp_{2n}$
where $n \geqslant 3$.
\begin{equation}\label{geom2}
\begin{tabular}{||p{4.3in}}
\emph{Let $M_i$ be the quotient
of $\X_i$ by a $(\G_i , K)$-arithmetic subgroup of
$\mathcal{G}_i$ for $i = 1, 2$. If $\G_1$ and $\G_2$ are twins then
\[
\Q \cdot L(M_2) = \lambda \cdot \Q \cdot L(M_1) \ \ \text{where} \ \
\lambda = \sqrt{\frac{2n+2}{2n - 1}}.
\]}
\end{tabular}
\end{equation}
(We refer to \cite{PrRap:weakly}, \S~1, for the notion of
arithmeticity and the explanation of other terms used here.) We
finally note that even though one can make $\X_1$ and $\X_2$ length-commensurable by scaling the metric on one
of them, this will never make them isospectral \cite{Yeung:isospectral}.

\section{Proofs of Proposition \ref{tori.old} and Theorem \ref{rank2}}\label{S:Compl}

\begin{proof}[Proof of Proposition \ref{tori.old}] We can assume
that $G_1$ and $G_2$ are connected absolutely almost simple
\emph{adjoint} $K$-groups having the same isogeny classes of maximal
$K$-tori. Assume that provisions (2) and (3) of the proposition do
not hold; let us show that (1) must hold. First, by \cite[Theorem
7.5]{PrRap:weakly}, $G_1$ and $G_2$ have the same Killing-Cartan
type. Furthermore, if $L_i$ is the minimal Galois extension of $K$
over which $G_i$ becomes an inner form then $L_1 = L_2$; in other
words, $G_1$ and $G_2$ are inner twists of the \emph{same}
quasi-split $K$-group. So, the required assertion is a consequence
of the following lemma.
\end{proof}

\begin{lem}\label{odd.similar}
Let $G_1$ and $G_2$ be connected absolutely almost simple \underline{adjoint}
$K$-groups of the same Killing-Cartan type which is different from
$\tA_{\ell}$ $(\ell > 1)$, $\tD_{2\ell + 1}$ $(\ell > 1)$ or $E_6$.
Assume that $G_1$ and $G_2$ are inner twists of the same quasi-split
$K$-group (which holds automatically if $G_1$ and $G_2$
are not of type $\tD$). If $G_1$ and $G_2$ have the same isogeny
classes of maximal $K$-tori then $G_1 \simeq G_2$.
\end{lem}

\begin{proof}
First, suppose that the groups are not of type $\tD$. As we have
seen in \S~\ref{ffield}, the fact that $G_1$ and $G_2$ have the same
isogeny classes of maximal $K$-tori implies that $\rk_{K_v}
  G_1 = \rk_{K_v}  G_2$ for all $v \in V^K$. For groups of
one of the types under consideration, this implies that $G_1 \simeq
G_2$ over $K_v$ for all $v \in V^K$ and then our assertion follows
from the Hasse principle for Galois cohomology of adjoint groups
(see \cite[\S~6]{PrRap:weakly} for details of the argument).

Now, suppose the groups are of type $\tD_{2\ell}$ for some $\ell
\geqslant 2$. There exists a maximal $K$-torus $T_1$ of $G_1$ that
is generic and such that $\rk_{K_v}   T_1 =
\rk_{K_v}   G_1$ at every place $v$ where at least one of
$G_1$ or $G_2$ is not quasi-split. (Note that the set of such $v$'s
is finite, cf.\ \cite[Theorem 6.7]{PlatRap}.) By hypothesis, $T_1$ is
isogenous to a maximal $K$-torus $T_2$ of $G_2$, which is
necessarily also generic. Following Lemma 4.3 and Remark 4.4 in
\cite{PrRap:weakly}, one finds a $K$-isomorphism $T_1 \to T_2$ that
extends to a $\bar{K}$-isomorphism $G_1 \to G_2$. Then our assertion
follows from Theorem 20 in \cite{G:outer}.
\end{proof}

\begin{proof}[Proof of Theorem \ref{rank2}] The ``if'' direction
is actually contained in Corollary \ref{C:BC-twins}---see Remark
\ref{R:L=2}. For the ``only if'' direction, we first observe that if
$G_1$ and $G_2$ have the same isomorphism classes of maximal
$K$-tori then by Lemma \ref{odd.similar} the groups $\SO(q_1)$ and
$\SO(q_2)$ are isomorphic, hence the forms $q_1$ and $q_2$ are
similar, yielding assertion (1). Thus, we can assume that $G_1 =
\SO(q)$ and $G_2 = \Spin(q)$ for a single quadratic form $q$.

To prove assertion (2), it is enough to show that if $v \in V^K$ is
such that the Witt index of $q$ over $K_v$ is 1 then there exists a
2-dimensional $K_v$-torus $T_1$ that has a $K_v$-embedding into
$G_1$ but does not allow a $K_v$-embedding into $G_2$. For this we
pick a quadratic extension $L/K_v$ and set
$$
T_1 = \GL_1 \times \mathrm{R}^{(1)}_{L/K_v}(\GL_1).
$$
We can write $q = q' \perp q''$ where $q'$ is a hyperbolic plane.
Then $\SO(q') = \GL_1$ and $\SO(q'') = \mathrm{PSL}_{1 , D}$
where $D$ is a quaternion division algebra over $K_v$. Since $L$
embeds in $D$, the torus $\mathrm{R}^{(1)}_{L/K_v}(\GL_1)$
embeds in $SL_{1 , D}$ and then also in $\mathrm{PSL}_{1 , D}$. It
follows that $T_1$ embeds in $G_1 = \SO(q)$. On the other hand, let
$T_2 \subset G_2$ be a maximal $K_v$-torus that splits over $L$. We
can identify $G_2$ with $\mathrm{SU}(A , \tau)$ where $A = M_2(D)$
with $D$ a quaternion division algebra over $K$ and $\tau$ is a
symplectic involution on $A$. Let $E_2$ be the $K_v$-subalgebra of
$A$ generated by $T_2(K_v)$. Then $E_2 \otimes_{K_v} L \simeq L^4$.
As in \S~\ref{S:real}, we conclude that $(E_2 , \tau \vert_{E_2})$ is
isomorphic to $(L , \sigma) \times (L , \sigma)$ where $\sigma$ is
the nontrivial automorphism of $L$, or to $(L \times L , \lambda)$
where $\lambda$ is the switch involution. Then $T_2 =
\mathrm{SU}(E_2 , \tau \vert_{E_2})$ is isomorphic respectively to
$\mathrm{R}^{(1)}_{L/K_v}(\GL_1)^2$ or
$\mathrm{R}_{L/K_v}(\GL_1)$. Neither such torus can be
isomorphic to $T_1$.
\end{proof}

\section{Alternative proofs via Galois cohomology}

Although the main body of the paper demonstrates the effectiveness
(and in fact the ubiquity)  of the technique of \'etale algebras in
dealing with maximal tori of classical groups, it is worth pointing
out that some parts of the argument can also be given in the
language of Galois cohomology of algebraic groups. In this section, we will illustrate such an exchange by giving a cohomological proof of the ``if'' direction of Theorem \ref{BC}(2), i.e., of Corollary \ref{C:BC-twins}\eqref{twin.same}.  

Our main tool is Proposition \ref{A2}, for which we need some notation.
Let $G$ be a connected semi-simple algebraic group
over a number field $K$. Fix a maximal $K$-torus $T$ of $G$, and let
$N = N_G(T)$ and $W = N/T$ denote respectively its normalizer and
the corresponding Weyl group. For any field extension $P/K$, we let
$\theta_P \colon H^1(P , N) \to H^1(P , W)$ denote the map induced
by the natural $K$-morphism $N \to W$, and let
$$
\sC(P) := \Ker \left(H^1(P , N) \longrightarrow H^1(P
, G) \right);
$$
its elements are in
one-to-one correspondence with the $G(P)$-conjugacy classes of
maximal $P$-tori in $G$, see for example \cite[Lemma 9.1]{PrRap:weakly} where this
correspondence is described explicitly.
 There is an obvious
$K$-defined map $W \to \mathrm{Aut}\: T$, so for any $\xi \in H^1(K
, W)$ one can consider the corresponding twisted $K$-torus
${}_{\xi}T$.

\begin{prop} \label{A2}
 Assume that there exists a
subset $V_0 \subset V_{\infty}^K$ such that $G$ is $K_v$-anisotropic
for all $v \in V_0$ and is $K_v$-split for all $v \in V^K \setminus
V_0$. Then the sequence
\begin{equation} \label{A2.1}
\begin{CD}
\sC(K) @>{\theta_K}>> H^1(K, W) @>{\prod \rho_v}>>
\prod_{v \in V_0} H^1(K_v , W)
\end{CD}
\end{equation}
is exact.
\end{prop}
Here $\rho_v$ denotes the natural restriction map $H^1(K, W) \ra H^1(K_v, W)$.

\begin{proof}
If $V_0$ is empty then
it follows from the Hasse principle for adjoint groups \cite[Theorem
6.22]{PlatRap} that $G$ is $K$-split. In this case it was shown by Gille
\cite{Gille:tori} and Raghunathan \cite{Ragh:tori} (or earlier by Kottwitz \cite{Kott:rational}) that
$\theta_K(\sC(K)) = H^1(K , W)$, and our claim follows. So, we will assume in the rest
of the argument that $V_0$ is not empty.

We first prove that $\rho_v \theta_K = 0$ for all $v \in V_0$.
Given $\xi \in \sC(K)$, one can pick $g \in
G(\bar{K})$ such that $n(\sigma) := g^{-1} \sigma(g)$ belongs to
$N(\bar{K})$ for all $\sigma \in \mathrm{Gal}(\bar{K}/K)$, and the
cocycle $\sigma \mapsto n(\sigma)$ represents $\xi$. Then the maximal torus
$T' = gTg^{-1}$ is defined over $K$. Now, let $v \in V_0$. According
to our definitions, $G$ is anisotropic over $K_v = \mathbb{R}$, so
it follows from the conjugacy of maximal tori in compact Lie groups
that $T$ and $T'$ are conjugate by an element of $G(K_v)$. Then the
one-to-one correspondence between the elements of $\sC(K_v)$
and the $G(K_v)$-conjugacy classes of maximal $K_v$-tori in $G$ (or
a simple direct computation) implies that the image of $\xi$ under
the restriction map $\sC(K) \to \sC(K_v)$ is
trivial, and hence the image of $\theta_K(\xi)$ under the
restriction map $H^1(K , W) \to H^1(K_v , W)$ is trivial as well. 

Now suppose that $G$ is simply connected; we verify that every $\xi \in \cap_{v\in V_0} \ker \rho_v$ is in the image of $\theta_K$.
Pick $v \in V_0$.
Since $\xi$ lies in the kernel of $H^1(K , W) \to H^1(K_v , W)$, the
twisted torus ${}_{\xi}T$ is $K_v$-isomorphic to $T$, hence $K_v$
anisotropic (as $G$ is $K_v$-anisotropic). Thus, 
\[
\Ker \left(H^2(K , {}_{\xi}T) \to \prod\nolimits_{v \in V^K} H^2(K_v ,
{}_{\xi}T)\right) = 0
\] 
by \cite[Prop.~6.12]{PrRap:weakly}.
Invoking now \cite[Th.~9.2]{PrRap:weakly}, we see that
to prove the inclusion $\xi \in \theta_K(\sC(K))$, it is
enough to show that $\rho_v(\xi) \in \theta_{K_v}(\sC(K_v))$
for all $v \in V^K$. If $v \in V_0$ then by construction
$\rho_v(\xi)$ is trivial, and there is nothing to prove. Otherwise,
the group $G$ is $K_v$-split, so by the result of
Gille-Kottwitz-Raghunathan we have $\theta_{K_v}(\sC(K_v)) =
H^1(K_v , W)$, and the inclusion $\rho_v(\xi) \in
\theta_{K_v}(\sC(K_v))$ is obvious. Since $\xi$ was arbitrary, we have proved that $\cap \ker \rho_v$ contains the image of $\theta_K$.

In case $G$ is not simply connected, we fix a $K$-defined universal cover $\pi \colon \Gt \to G$ of $G$ and use the tilde
to denote the objects associated with $\Gt$. Then $\pi$ yields
a $K$-isomorphism of $\Wt$ and $W$ and we have a commutative diagram
\[
\begin{CD}
\widetilde{\sC}(K) @>{\widetilde{\th}_K}>> H^1(K, \Wt) @>{\prod \widetilde{\rho}_v}>> \prod_{v \in V_0} H^1(K_v, \Wt) \\
@VVV @| @| \\
\sC(K) @>{\th_K}>> H^1(K, W) @>{\prod \rho_v}>> \prod_{v \in V_0} H^1(K_v, W).
\end{CD}
\]
The top row is exact by the previous paragraph, hence $\cap \ker \rho_v$ contains the image of $\theta_K$.
\end{proof}

We now begin to work our way towards the proof of Theorem \ref{BC}(2)/Corollary \ref{C:BC-twins}\eqref{twin.same}. Let
$G_1$ be adjoint of type $\tB_\ell$ and let $G_2$ be simply connected of type $\tC_\ell$ for some $\ell \ge 2$.  We will use a
subscript $i \in \{ 1 , 2 \}$ to denote the objects associated with
$G_i$. In particular, we let $T_i$ denote a maximal torus of $G_i$,
and let $N_i = N_{G_i}(T_i)$ and $W_i = N_i/T_i$ be its normalizer
and the Weyl group. Then $W_i$ naturally acts on $T_i$ by
conjugation. We say that the morphisms of algebraic groups $\varphi
\colon T_1 \to T_2$ and $\psi \colon W_1 \to W_2$ are {\it
compatible} if
$$
\varphi(w \cdot t) = \psi(w) \cdot \varphi(t) \quad \text{for all 
$t \in T_1, \ w \in W_1$}.
$$

\begin{lem} \label{A4}
One can pick maximal $K$-tori $T_i$
of $G_i$ for $i = 1, 2$ so that there exist compatible $K$-defined
isomorphisms
$$
\varphi \colon T_1 \to T_2 \ \ \ \text{\it and} \ \ \ \psi \colon
W_1 \to W_2.
$$
\end{lem}

\begin{proof}
Imitating the argument given in \cite[Proposition 6.16]{PlatRap}, it is
easy to see that there exists a quadratic extension $L/K$ that
splits {\it both} $G_1$ and $G_2$. Indeed, let $V_i$ be the (finite)
set of places $v \in V^K$ such that $G_i$ does not split over $K_v$,
and let $V = V_1 \cup V_2$. Pick a quadratic extension $L/K$ so that
the local degree $[L_w : K_v] = 2$ for all $v \in V$ and $w \vert
v$. We claim that $L$ is as required. By the Hasse principle, it is
enough to show that both $G_1$ and $G_2$ split over $L_w$ for any $w
\in V^L$. For a given $w$, we let $v \in V^K$ be the place that lies
below $w$. If $v \notin V$ then by our construction $G_1$ and $G_2$
split already over $K_v$, and there is nothing to prove. If $v \in
V$ then $[L_w : K_v] = 2$, and then the proof of \cite[Proposition
6.16]{PlatRap} that $G_1$ and $G_2$ split over $L_w$, as required. 

Now,
let $\sigma \in \mathrm{Gal}(L/K)$ be a generator. According to
\cite[Lemma 6.17]{PlatRap}, for each $i \in \{ 1 , 2 \}$, there exists
an $L$-defined Borel subgroup $B_i$ of $G_i$ such that $T_i := B_i
\cap B_i^{\sigma}$ is a maximal $K$-torus of $G_i$ that splits over
$L$. Considering the action of $\sigma$ on the root system $\Phi(G_i
, T_i)$, we see that it takes the system of positive roots
corresponding to $B_i$ into the system of negative roots. For groups
of types $\textsf{B}_{\ell}$ and $\textsf{C}_{\ell}$, this implies
that $\sigma$ acts on the character group $X(T_i)$ as multiplication
by $(-1)$. It easily follows from the description of the
corresponding root systems (cf.\ \cite{Bou:g4}) that there exist
compatible (in the obvious sense) isomorphisms $\varphi^* \colon
X(T_2) \to X(T_1)$ (of abelian groups) and $\psi \colon W_1 \to W_2$
(of abstract groups considered as subgroups of $GL(X(T_1))$ and
$GL(X(T_2))$). Then $\varphi^*$ gives rise to an isomorphism
$\varphi \colon T_1 \to T_2$ of algebraic groups that is compatible
(as defined above) with $\psi$ (which can be considered as a
morphism of algebraic groups). It remains to observe that since
$\sigma$ acts on $X(T_1)$ and $X(T_2)$ as multiplication by $(-1)$,
both $\varphi$ and $\psi$ are $K$-defined (in fact, $\sigma$ acts on
$W_1$ and $W_2$ trivially).
\end{proof}

\begin{rmk*} If both groups $G_1$ and $G_2$ are $K$-split
then one can, of course, take for  $T_1$ and $T_2$ their maximal
$K$-split tori.
\end{rmk*}

For the rest of the paper, we fix compatible $K$-defined
isomorphisms
$$
\varphi^0 \colon T_1^0 \to T_2^0 \ \ \text{and} \ \ \psi^0 \colon
W_1^0 \to W_2^0.
$$
(Thus, we henceforth slightly change the notations used in Lemma
\ref{A4}.) Given {\it arbitrary} maximal $K$-tori $T_i$ of $G_i$ for $i
= 1, 2$, we pick elements $g_i \in G(\bar{K})$ so that
$$
T_i = g_i T_i^0 g_i^{-1},
$$
and then for any $\sigma \in \mathrm{Gal}(\bar{K}/K)$, the element
$n_i(\sigma) := g_i^{-1} \sigma(g_i)$ belongs to $N_i^0(\bar{K})$.
Let $\varphi = \varphi(g_1 , g_2)$ be the morphism $T_1 \to T_2$
defined by
$$
\varphi(t) = g_2 \varphi^0(g_1^{-1} t g_1) g_2^{-1},
$$
and let $\nu_i^0 \colon N_i^0 \to W_i^0$ denote the canonical
morphism.

\begin{lem} \label{A5}
 If
\begin{equation}\label{E:Append2}
\psi^0(\nu_1^0(n_1(\sigma))) = \nu_2^0(n_2(\sigma)) \ \ \text{for
all} \ \ \sigma \in \mathrm{Gal}(\bar{K}/K)
\end{equation}
then $\varphi = \varphi(g_1 , g_2)$ is defined over $K$.
\end{lem}

\begin{proof}
We need to show that $\varphi$ commutes with every $\sigma \in
\mathrm{Gal}(\bar{K}/K)$. Since $\varphi^0$ is defined over $K$, for
any $t \in T_1(\bar{K})$, we have
\begin{align*}
\sigma(\varphi(t)) &= \sigma(g_2) \varphi^0(\sigma(g_1)^{-1}
\sigma(t) \sigma(g_1))\sigma(g_2)^{-1}\\
&= g_2 n_2(\sigma) \varphi^0(n_1(\sigma)^{-1} g_1^{-1} \sigma(t) g_1
n_1(\sigma)) n_2(\sigma)^{-1} g_2^{-1}\\
&= g_2 \, [ \,(\nu_2^0(n_2(\sigma))) \cdot
\varphi^0((\nu_1^0(n_1(\sigma))) \cdot (g_1^{-1} \sigma(t) g_1)) \,
] \, g_2^{-1}.
\end{align*}
Since $\varphi^0$ is compatible with $\psi^0$, condition
(\ref{E:Append2}) implies that the latter reduces to
$$
g_2 \varphi^0(g_1^{-1} \sigma(t) g_1) g_2^{-1} = \varphi(\sigma(t)).
$$
It follows that $\sigma(\varphi(t)) = \varphi(\sigma(t))$, i.e.
$\varphi$ commutes with $\sigma$, as required.
\end{proof}

Pursuant to the above notations, for an extension $P/K$ and $i = 1,
2$, we set
$$
\sC_i(P) = \Ker \left(H^1(P , N_i^0) \to H^1(P ,
G_i)\right),
$$
and let $\theta_{i P} \colon H^1(P, N_i^0) \to H^1(P , W_i^0)$
denote the canonical map (induced by $\nu_i$). The isomorphism
$H^1(K , W_1^0) \to H^1(K , W_2^0)$ induced by $\psi^0$ will still
be denoted by $\psi^0$.

\begin{lem} \label{A6}
Assume that
\begin{equation}\label{E:Append3}
\psi^0(\sC_1(K)) = \sC_2(K).
\end{equation}
Then for $i = 1$ or $2$, given any maximal $K$-torus $T_i$ of $G_i$
and an element $g_i \in G_i(\bar{K})$ such that $T_i = g_iT_i^0
g_i^{-1}$, there exists $g_{3 - i} \in G_{3 - i}(\bar{K})$ such that
the maximal torus
$$
T_{3-i} := g_{3-i} T_{3-i}^0 g_{3 - i}^{-1}
$$
and the isomorphism
$$
\varphi(g_1 , g_2) \colon T_1 \to T_2
$$
are $K$-defined. Thus, in this case $G_1$ and $G_2$ have the same
isomorphism classes of maximal $K$-tori.
\end{lem}

\begin{proof}
To keep our notations simple, we will give an argument for $i = 1$
(the argument in the case $i = 2$ is totally symmetric). As above,
we set $n_1(\sigma) = g_1^{-1} \sigma(g_1) \in N_1^0(\bar{K})$ for
$\sigma \in \mathrm{Gal}(\bar{K}/K)$, observing that these elements
define a cohomology class $n_1 \in \sC_1(K).$ Then
\eqref{E:Append3} implies that there exists $h_2 \in G_2(\bar{K})$
such that for the cohomology class $m_2 \in \sC_2(K)$
defined by the elements $m_2(\sigma) = h_2^{-1} \sigma(h_2) \in
N_2^0(\bar{K})$, we have
$$
\psi^0(\theta_{1 K}(n_1)) = \theta_{2 K}(m_2) \ \ \text{in} \ \
H^1(K , W_2).
$$
Then there exists $w_2 \in W_2(\bar{K})$ such that
\begin{equation}\label{E:Append4}
\psi^0(\nu_1^0(n_1(\sigma))) = w_2^{-1} \nu_2^0(m_2(\sigma))
\sigma(w_2) \ \ \text{for all} \ \ \sigma \in
\mathrm{Gal}(\bar{K}/K).
\end{equation}
Picking $z_2 \in N_2^0(\bar{K})$ so that $\nu_2^0(z_2) = w_2$, and
setting
$$
g_2 = h_2z_2 \quad \text{and} \quad n_2(\sigma) = g_2^{-1} \sigma(g_2)
\in N_2^0(\bar{K}) \quad \text{for $\sigma \in
\mathrm{Gal}(\bar{K}/K)$},
$$
we obtain  from \eqref{E:Append4} that \eqref{E:Append2} holds. Then
$g_2$ is as required. Indeed, the fact that $n_2(\sigma) \in
N_2^0(\bar{K})$ implies that $T_2 = g_2 T_2^0 g_2^{-1}$ is defined
over $K$, and Lemma \ref{A5} yields that the morphism $\varphi(g_1 , g_2)
\colon T_1 \to T_2$ is also defined over $K$.
\end{proof}

\begin{proof}[Proof of Corollary \ref{C:BC-twins}\eqref{twin.same}] 
Suppose that $G_1$ and $G_2$
are twins, and let $V_0$ be the set of all archimedean places $v \in
V^K$ such that $G_1$ and $G_2$ are both $K_v$-anisotropic. Then for
any $v \in V^K \setminus V_0$, both $G_1$ and $G_2$ are $K_v$-split.
Then according to Proposition \ref{A2} we have
\[
\theta_{i K}(\sC_i(K)) = \ker \left( H^1(K, W_i^0) \ra \prod\nolimits_{v \in V_0} H^1(K_v, W_i^0) \right)
\]
for $i = 1, 2$, and as $\psi_0 \!: W_1^0 \ra W_2^0$ is an isomorphism, 
condition (\ref{E:Append3})
holds, and the claim follows from Lemma \ref{A6}.
\end{proof}

\begin{rmk*}
It follows from the explicit description of
the root systems of types $\textsf{B}_{\ell}$ and
$\textsf{C}_{\ell}$ that the isomorphism $\varphi$ in Lemma \ref{A4} can
be chosen so that for $t \in T_1(\bar{K})$ there exist $\lambda_1,
\ldots , \lambda_{\ell} \in \bar{K}^{\times}$ such that the values
of the roots $\alpha \in \Phi(G_1 , T_1)$ on $t$ are
$$
\lambda_i^{\pm 1}, \ \ i = 1, \ldots , \ell, \ \ \text{and} \ \
\lambda_i^{\pm 1} \cdot \lambda_j^{\pm 1}, \ \ i , j = 1, \ldots ,
\ell, \ i \neq j,
$$
and the values of the roots $\alpha \in \Phi(G_2 , T_2)$ on
$\phi(t)$ are
$$
\lambda_i^{\pm 2}, \ \ i = 1, \ldots , \ell, \ \ \text{and} \ \
\lambda_i^{\pm 1} \cdot \lambda_j^{\pm 1}, \ \ i , j = 1, \ldots ,
\ell, \ i \neq j.
$$
Then any identification of the form $\varphi(g_1 , g_2)$ also has
this property, which  was used in \cite{PrRap:fields}.

Alternatively, suppose that $G_i$ for $i = 1, 2$ is realized as
$\mathrm{SU}(A_i , \tau_i)$  as described in the beginning of \S\ref{S:BC}.
Let $E_1$ be a $(\tau_1 \otimes \mathrm{id}_{\bar{K}})$-invariant
maximal commutative \'etale $\bar{K}$-subalgebra of $A_1 \otimes_K
\bar{K}$ satisfying (2.2), and let $\sigma_1 = \tau_1 \vert_{E_1}$.
Then in the notations of \S 6, the algebra $(E'_1 , \sigma'_1)$
admits a $\bar{K}$-embedding embedding into $(A_2 \otimes_K \bar{K}
, \tau_2 \otimes \mathrm{id}_{\bar{K}})$, and we let $(E_2 ,
\sigma_2)$ the image of this embedding. It is easy to see that if we
let $T_i$ denote the maximal torus of $G_i$ defined by $(E_i ,
\sigma_i)$ then the isomorphism $T_1 \simeq T_2$ coming from the
isomorphism of algebras $(E'_1 , \sigma'_1) \simeq (E_2 , \sigma_2)$
is the same as the isomorphism coming from the description of the
root systems (cf.~the proof of Lemma \ref{A4}); in particular, it is
compatible with the natural isomorphism of the Weyl groups. So, the
assertion of Lemma \ref{A4} means that given {\it any} $K$-algebras with
involution $(A_1 , \tau_1)$ and $(A_2 , \tau_2)$ as above, there
exists a $\tau_1$-invariant maximal commutative \'etale
$K$-subalgebra $E_1$ of $A_1$ that satisfies \eqref{E:dimen} and is such that
for $\sigma_1 = \tau_1 \vert_{E_1}$, the algebra $(E'_1 , \sigma'_1)$
admits an embedding into $(A_2 , \sigma_2)$. Moreover, by Corollary \ref{C:BC-twins}\eqref{twin.same}, if the corresponding groups $G_1$ and $G_2$ are twins then the
correspondence $(E_1 , \sigma_1) \mapsto (E'_1 , \sigma'_1)$ gives a
bijection between the sets of isomorphism classes of maximal
commutative \'etale $K$-subalgebras of $(A_1 , \tau_1)$ and $(A_2 ,
\tau_2)$ that are invariant under the respective involutions and
satisfy \eqref{E:dimen}. Thus, we recover Proposition \ref{P:same-subalg}.
\end{rmk*}

\medskip
\noindent{\small{\textbf{Acknowledgements.} We thank the referee for their helpful comments.  SG's research was
partially supported by NSA grant H98230-11-1-0178 and the Charles T.~Winship Fund.  AR's
research was partially supported by NSF grant DMS-0965758 and the
Humboldt Foundation. During the preparation of the final version of
this paper, he  was visiting the Mathematics Department of the
University of Michigan as a Gehring Professor; the hospitality and
generous support of this institution are gratefully acknowledged. }}

\bibliographystyle{amsalpha}
\bibliography{skip_master}

\end{document}